\documentclass[10pt,reqno]{amsart}
\usepackage{amsmath,amssymb, graphicx}
\usepackage[margin=1in]{geometry}
\usepackage{ifthen}
\usepackage{subcaption}
\usepackage{caption}
\usepackage{sidecap}
\usepackage{tikz}
\usetikzlibrary{positioning,fit,arrows.meta,shapes,matrix}
\usepackage{float}
\usepackage{colortbl}
\definecolor{black}{rgb}{0.0, 0.0, 0.0}
\definecolor{red}{rgb}{1.0, 0.5, 0.5}
\provideboolean{shownotes} 
\setboolean{shownotes}{true} 
\newcommand{\margnote}[1]{
\ifthenelse{\boolean{shownotes}}%
{\marginpar{\raggedright\tiny\texttt{#1}}}%
{}%
}
\newcommand{\hole}[1]{
\ifthenelse{\boolean{shownotes}}%
{\begin{center} \fbox{ \rule {.25cm}{0cm} \rule[-.1cm]{0cm}{.4cm}
\parbox{.85\textwidth}{\begin{center} \texttt{#1}\end{center}} \rule
{.25cm}{0cm}}\end{center}} {} }

\title[Synchronization of Kuramoto oscillators with time delay in networks]{Asymptotic synchronization of Kuramoto oscillators with time delay and non-universal interaction}

\author[Carty]{Conor Carty}
\address[Conor Carty]{\newline Department of Mathematics, Yonsei University,
    \newline Seoul 03722, Republic of Korea}
\email{conorcarty@gmail.com}

\author[Choi]{Young-Pil Choi}
\address[Young-Pil Choi]{\newline Department of Mathematics, Yonsei University,
    \newline Seoul 03722, Republic of Korea}
\email{ypchoi@yonsei.ac.kr}

\author[Cicolani]{Chiara Cicolani}
\address[Chiara Cicolani]{\newline Dipartimento di Ingegneria e Scienze dell'Informazione e Matematica, Universit\`a di L'Aquila,
\newline 67100 L'Aquila, Italy }
\email{chiara.cicolani@graduate.univaq.it}

\author[Pignotti]{Cristina Pignotti}
\address[Cristina Pignotti]{\newline Dipartimento di Ingegneria e Scienze dell'Informazione e Matematica, Universit\`a di L'Aquila,
\newline 67100 L'Aquila, Italy }
\email{cristina.pignotti@univaq.it}

\numberwithin{equation}{section}

\newtheorem{theorem}{Theorem}[section]
\newtheorem{lemma}{Lemma}[section]

\newtheorem{remark}{Remark}[section]
\newtheorem{definition}{Definition}[section]
\newtheorem{maintheorem}{Theorem}

\def\({\begin{eqnarray}}
\def\){\end{eqnarray}}
\def\[{\begin{eqnarray*}}
\def\]{\end{eqnarray*}}

\newcommand{\mc}{\mathcal C}

\newcommand{\bq}{\begin{equation}}
\newcommand{\eq}{\end{equation}}

\newcommand{\lt}{\left}
\newcommand{\rt}{\right}

\def\N{\mathbb{N}}

\begin{document}
\allowdisplaybreaks



\begin{abstract}
We study the emergence of synchronization in the Kuramoto model on a digraph in the presence of time delays. Assuming the digraph is strongly connected, we first establish a uniform bound on the phase diameter and subsequently prove the asymptotic frequency synchronization of the oscillators under suitable assumptions on the initial configurations. In the case of an all-to-all connection, we obtain an exponential synchronization estimate. Additionally, we present numerical simulations, providing further insights into the synchronization and oscillatory behaviors of the oscillator frequencies depending on the network structure and the magnitude of the time delay.
\end{abstract}

\maketitle \centerline{\date}


%
%
%
%

%
\section{Introduction}

Synchronization phenomena appear in nature, ranging from the rhythmic flashing of fireflies to the coordinated beating of the heart muscle cells. Understanding the mechanisms underlying such spontaneous synchronization has been a subject of considerable interest across various scientific disciplines, including physics, biology, and engineering \cite{BB66, Str03, PRK03, Wi67}. One of the seminal models that has greatly contributed to our understanding of synchronization is the Kuramoto model \cite{Ku84}; it provides a simple yet powerful framework for studying the emergence of synchronization in large ensembles of coupled oscillators. Over the years, the Kuramoto model has found applications in various fields, including neuroscience \cite{BHD10, ZLQYWL22} and power grids \cite{DB12, CL19} among others. Its mathematical tractability and versatility make it a valuable tool for theoretical investigations as well as for interpreting experimental observations of synchronization phenomena.

In the present work, we analyze the emergence of synchronization in the Kuramoto model on directed graphs, considering the effects of time delays in the coupling between oscillators. More precisely, we investigate the conditions under which synchronization emerges in such systems and explore the influence of network topology and time delays on the synchronization dynamics.  To present our model, we consider a digraph $ \mathcal{G}=(\mathcal{V}, \mathcal{E})$ consisting of a finite set $\mathcal{V}=\{1,\dots,N \}$ of vertices and a set $\mathcal{E} \subset \mathcal{V} \times \mathcal{V}$ of arcs. We assume that Kuramoto oscillators are located at vertices and interact with each other via the underlying network topology. For each vertex $i$, let $\mathcal{N}_i$ be the set of vertices that directly influence the vertex $i,$ namely $\mathcal{N}_i= \left\{\, j\in \{1, \dots, N\} :  (i,j)\in \mathcal{E}\,\right\}$, and we denote by  $N_i$ the number of oscillators in the set $\mathcal{N}_i,$ i.e. $N_i=\left\vert \mathcal{N}_i\right\vert$. Then,  our main system reads 
\begin{equation}\label{main_eq2}
\frac{d}{dt}\theta_i(t) = \Omega_i + \frac\kappa {N-1} \sum_{j\ne i} \chi_{ij} \sin(\theta_j(t - \tau_{ij}) - \theta_i(t)), \quad i = 1,\dots, N, \quad  t > 0,
\end{equation}
where $\theta_i = \theta_i(t)$ represents the phase of $i$-th Kuramoto oscillator at vertex $i$ and time $t>0$, and $\tau_{ij}>0$ denotes communication time delay in the information flow from vertex $j$ to vertex $i$. Here the network topology is given by its $(0,1)$-adjacency matrix $(\chi_{ij})$ defined as
\begin{equation*} \chi_{ij}=\begin{cases} 1 \ \mbox{if $j$ transmits information to $i$}, \\
0 \ \mbox{otherwise}. \end{cases} 
\end{equation*}
Notice that self-time delay is not allowed, i.e. $\tau_{ii}=0$ and we exclude a self-loop, i.e. $i \notin \mathcal{N}_i$ for all $i =1,\dots, N$.  

The system \eqref{main_eq2} describes the dynamic interactions among the oscillators, accounting for both their natural frequencies $\Omega_i$, which is assumed to be a random variable extracted from a given distribution $g = g(\Omega)$, and the influence of neighboring oscillators weighted by the coupling strength $\kappa > 0$ and communication delays $\tau_{ij}$. The classical Kuramoto model, extensively studied in the literature (see, e.g. \cite{ABPRS05, BS00, CCHKK14, CCP19, CHJY12, CS09, DB11, HR20, MS07, Str00}), corresponds to the cases where $\chi_{ij} = 1$ and $\tau_{ij} = 0$ for all $i,j=1,\dots, N$. However, in our formulation \eqref{main_eq2}, we introduce two additional structures, namely the time delay effect and network topology. 

Incorporating time-delayed mathematical formulations is crucial for capturing the dynamics of real-world systems accurately. Time delays are prevalent in various physical and biological systems, stemming from finite signal propagation speeds, processing times, or communication latencies. Neglecting these delays can lead to inaccurate predictions and overlook essential aspects of system behavior. Additionally, taking into account the network topologies in the interaction is very natural. The connectivity structure dictates which oscillators directly influence each other's phases. In \eqref{main_eq2}, the network connectivity determines the phase relations between connected oscillators and influences the overall synchronization behavior. For recent results and more detailed background regarding these extensions, we refer to \cite{CL19, CP21, DB12, DHK20, HLX13, ZZ23, Zhu22, ZhuNHM} and references therein. 

We supply the system \eqref{main_eq2} with the initial data:
\begin{equation}\label{IC2}
\theta_i(s) = \theta^0_i(s) \quad \mbox{for} \quad s \in [-\tau,0],
\end{equation}
where $\theta_i^0\in \mc^1_b(-\tau, 0) \cap \mc[-\tau,0],$ $ i=1,\dots, N$, and  $\tau :=\max_{i,j=1, \dots, N}\, \tau_{ij}$.

In the current work, we focus on the emergence of frequency synchronization in \eqref{main_eq2} with some interaction network and time delay, namely we identify the class of initial configurations leading to that all oscillators tend to have the same frequency as time goes to infinity.  To begin with, we introduce some basic concepts related to the directed graphs. A \emph{path} in a digraph $\mathcal{G}$ from $i_0$ to $i_p$ is a finite sequence $i_0, i_1, \dots, i_p$ of distinct vertices such that each successive pair of vertices is an arc of $\mathcal{G}.$ The integer $p$ is called \emph{length} of the path. If there exists a path from $i$ to $j$, then vertex $j$ is said to be \emph{reachable} from vertex $i$ and we define the distance from $i$ to $j$, $d(i,j)$, as the length of the shortest path from $i$ to $j$. A digraph $\mathcal{G}$ is said to be \emph{strongly connected} if each vertex can be reachable from any vertex. We call the \emph{depth} $\gamma$ of the digraph the quantity:
\begin{equation*}
 \gamma := \max_{1 \leq i,j \leq N} d(i,j). 
\end{equation*}
Thus, any agent can be connected to one other via no more than $\gamma$ intermediate agents. It is trivial that $\gamma \leq N-1.$

We next introduce some notations to be used throughout the paper. For a solution $\theta(t) := (\theta_1(t), \cdots, \theta_N(t))$ to \eqref{main_eq2}, we denote the phase and velocity diameters by
\[
d_\theta(t) :=\max_{1 \leq i,j \leq N}|\theta_i(t) - \theta_j(t)| \quad \mbox{and} \quad d_\omega(t) :=\max_{1 \leq i,j \leq N}|\omega_i(t) - \omega_j(t)|,
\]
respectively, where $\omega_i(t) := \dot\theta_i(t):=\frac{d\theta_i(t)}{dt}$ with $\omega_i(0) := \lim_{t\to 0-} \dot\theta_i(t).$ Along with those notations, we also set
\begin{equation*}
D(\Omega) :=\max_{1 \leq i,j \leq N}|\Omega_i - \Omega_j|,
\end{equation*}
\begin{equation}\label{Dzero}
D_\theta(0) :=\max_{1 \leq i,j \leq N, s,t\in [-\tau, 0]}|\theta_i(s) - \theta_j(t)|,\quad \mbox{and} \quad
D_\omega(0) :=\max_{1 \leq i,j \leq N, s,t\in [-\tau, 0]}|\omega_i(s) - \omega_j(t)|.
\end{equation}

We also provide a notion of our complete frequency synchronization in the definition below.
\begin{definition}\label{def_sync} Let $\theta(t) := (\theta_1(t), \dots, \theta_N(t))$ be the global classical solution to the system \eqref{main_eq2}-\eqref{IC2}. Then the system exhibits the complete frequency synchronization if and only if the velocity diameter tends to 0, as the time goes to infinity:
\[
\lim_{t \to \infty} d_\omega(t) = 0.
\]
\end{definition}

Let us now state our main asymptotic frequency synchronization results in a somewhat rough manner. The precise statements are given in Theorem \ref{theorem_main2} (strongly connected case) and Theorem \ref{theorem_main1} (all-to-all connected case).

\begin{maintheorem} Let $\{\theta_i(t)\}_{i=1}^N$ be a solution to the Kuramoto model \eqref{main_eq2}  with initial data \eqref{IC2}. We have the following asymptotic frequency synchronization results.
\begin{itemize}
\item[(i)] (strongly connected case) Assume the digraph $\mathcal{G}$ is strongly connected, i.e. each vertex can be reachable from any vertex. Suppose that $D_\theta(0) <\pi$, $\tau$ is sufficiently small, and $\kappa$ is large enough. Then the time-delayed Kuramoto oscillators \eqref{main_eq2} achieve the asymptotic complete frequency synchronization in the sense of Definition \ref{def_sync}. 

\item[(ii)] (all-to-all connected case) Assume that all oscillators are connected, i.e., $\chi_{ij} = 1$ for all $i,j=1,\dots, N$. Suppose that $D_\theta(0) <\pi$, $\tau$ is sufficiently small, and $\kappa$ is large enough. Then the time-delayed Kuramoto oscillators \eqref{main_eq2} achieve the asymptotic complete frequency synchronization in the sense of Definition \ref{def_sync} exponentially fast. 
\end{itemize}
\end{maintheorem}

The initial step of our analysis involves establishing a uniform-in-time bound estimate of the phase diameter. Specifically, we demonstrate that for sufficiently large times, the oscillators will be confined within a region of a quarter circle under suitable assumptions on the initial configurations. This bound estimate plays a critical role in unveiling the dissipation structure of \eqref{main_eq2}, and thus it serves as a foundational element for deriving the asymptotic frequency synchronization result. Notably, our strategy improves the previous work \cite{DHK20} where the initial phase diameter $D_\theta(0)$ is assumed to be less than $\pi/2$. In the case of all-to-all connection, i.e., the oscillators are all connected, we achieve a more robust result, namely an exponential asymptotic synchronization estimate. These results extend previous results in the literature. In particular, compared to \cite{CP21}, we introduce considerations for the network structure and pair-dependent delays. Moreover, we significantly relax previous assumptions on the time delay size (see e.g. \cite{CP21, ZhuNHM}).

The rest of this paper is organized as follows. In Section \ref{sec_bdd}, we consider the strongly connected case and provide that the diameter remains bounded by a constant less than $\pi/2$ after some finite time. Sections \ref{sec_syn_str} and \ref{sec_syn_all2} are devoted to presenting the detailed proofs of Theorem \ref{theorem_main2} (strongly connected case) and Theorem \ref{theorem_main1} (all-to-all connected case). Finally, in Section \ref{sec_numer}, we present numerical simulations, illustrating further insights into the synchronization and oscillatory behaviors of the oscillator frequencies depending on the network structure and the magnitude of the time delay.

%
%
%
%
\section{Uniform-in-time bound of phase diameter}\label{sec_bdd}
In this section, we provide the uniform-in-time bound estimate of the phase diameter. Specifically, our goal of this section is to obtain that the phase diameter $d_\theta(t)$ is bounded by for some $d_\infty \in (0,\frac\pi2)$ for any $t$ large enough under suitable assumptions on the initial configurations. For this, we first start with the estimate providing the bound on the difference between the time-delayed and non-time-delayed phases. Note that we can easily find
\[
|\omega_i(t)| \leq |\Omega_i| + \kappa \leq \max_{1 \leq i \leq N} |\Omega_i| + \kappa,
\]
for all $i =1,\cdots, N.$ Let us denote
\begin{equation}\label{erre}
R_\omega:=\max_{1 \leq i \leq N} |\Omega_i| + \kappa,
\end{equation}
then
\begin{equation}\label{bound_velocities}
\vert \omega_i(t)\vert \le R_\omega, \quad \forall\, t\ge 0.
\end{equation}
From \eqref{bound_velocities}, we deduce that for all
$i=1, \dots, N,$ 
\begin{equation*}
\vert \theta_i(t)-\theta_i(s)\vert = \left\vert \int_s^t \omega_i(r)dr\right\vert \le R_\omega\vert t-s\vert, \quad \forall \, s,t\ge 0.
\end{equation*}
In particular, for all $i,j=1, \dots, N,$ we have
\[
\vert\theta_i(t-\tau_{ij})-\theta_i(t)\vert \le R_\omega\tau, \quad t\ge 0.
\]

For the uniform-in-time bounded estimate of phase diameter $d_\theta(t)$, we need to use the dissipative structure of the system \eqref{main_eq2}. For this, motivated from \cite{ZZ23}, we define an ensemble of oscillators as a convex combination denoted by 
\begin{equation*}
\mathcal{L}_l^k\left(C_{l, k}\right):=\sum_{i=l}^k c_i \theta_i, \quad  1 \leq l \leq k \leq N,
\end{equation*}
where all $c_i$ are non-negative and $\sum_{i=l}^k c_i = 1$. We also introduce notions of root and general root in the definition below. 
\begin{definition}[Root and General Root]
\label{Root and General Root}\
\begin{itemize}
    \item[(i)] We say $\theta_k$ is a root if it is not influenced by any other oscillators, i.e. $i \notin \mathcal{N}_k$ for all  $i \neq k$.
    \item[(ii)] An ensemble of oscillators $\mathcal{L}_l^k\left(C_{l, k}\right)$ is a general root if it is not influenced by any oscillators excluded from the ensemble, i.e. $j \notin \mathcal{N}_i$ for all $i \in \{l,\ldots,k\}$ and $j \in\{1, \ldots, N\} \backslash\{l, \ldots, k\}$.
\end{itemize}
\end{definition}

For the analysis, we use the following algorithm, denoted by $\mathcal{A}$, proposed in \cite{HLZ20} for constructing convex combinations of oscillators. 

\medskip

\paragraph{\bf Step I} For any time $t$, we reorder the oscillator indices such that the phases are increasing from minimum to maximum:
\[
\theta_1 (t) \leq \theta_2 (t) \leq \cdots \leq \theta_{N} (t).
\]
For the next steps, we introduce the following sub-algorithms:
\begin{itemize}
    \item[$\left(\mathcal{A}_1\right)$:] If $\overline{\mathcal{L}}_k^{N}\left(\bar{C}_{k, N}\right)$ is not a general root, then we construct (from top to bottom)
$$
\overline{\mathcal{L}}_{k-1}^{N}\left(\bar{C}_{k-1, N}\right)=\frac{\bar{a}_{k-1} \overline{\mathcal{L}}_k^{N}\left(\bar{C}_{k, N}\right)+\theta_{k-1} }{\bar{a}_{k-1}+1} .
$$
    \item[$\left(\mathcal{A}_2\right)$:]  If $\underline{\mathcal{L}}_1^l\left(\underline{C}_{1, l}\right)$ is not a general root, then we construct (from  bottom to top)
$$
    \underline{\mathcal{L}}_1^{l+1}\left(\underline{C}_{1, l+1}\right)=\frac{\underline{a}_{l+1} \underline{\mathcal{L}}_1^l\left(\underline{C}_{1, l}\right)+\theta_{l+1} }{\underline{a}_{l+1}+1}.
$$
\end{itemize}

\medskip

\paragraph{\bf Step II} Since $\mathcal{G}$ is strongly connected, we have that $\overline{\mathcal{L}}_1^{N}\left(\bar{C}_{1, N}\right)$ is a general root, and $\overline{\mathcal{L}}_k^{N}\left(\bar{C}_{k, N}\right)$ is not a general root for $k>1$. Therefore, we may start from $\theta_{N}$ and follow the process $\mathcal{A}_1$ to construct $\overline{\mathcal{L}}_k^{N}\left(\bar{C}_{k, N}\right)$ until $k=1$.

\medskip

\paragraph{\bf Step III} Similarly, $\underline{\mathcal{L}}_1^{N}\left(\underline{C}_{1, N}\right)$ is a general root and $\underline{\mathcal{L}}_1^l\left(\underline{C}_{1, l}\right)$ is not a general root for $l<N$. Therefore, we may analogously start from $\theta_1$ and follow the process $\mathcal{A}_2$ until $l=N$.

\begin{remark}
The coefficients in the above constructions are determined inductively according to
\[
\begin{aligned}
& \overline{\mathcal{L}}_{k-1}^{N}\left(\bar{C}_{k-1, N}\right) \text{ with } \bar{a}_{N} =0, \  \bar{a}_{k-1} =\eta\left(2 N-k+2\right)\left(\bar{a}_k +1\right), \quad 2 \leq k \leq N, \\
& \underline{\mathcal{L}}_1^{k+1}\left(\underline{C}_{1, k+1}\right) \text{ with } \underline{a}_1 =0,\  \underline{a}_{k+1}=\eta\left(k+1+N\right)\left(\underline{a}_k +1\right), \quad 1 \leq k \leq N-1
\end{aligned}
\]
or in summation form:
$$\begin{aligned} 
& \bar{a}_{k-1} =\sum_{j=1}^{N-k+1} \eta^j P\left(2 N-k+2, j\right), \quad 2 \leq k \leq N, \\ 
& \underline{a}_{k+1} =\sum_{j=1}^k \eta^j P\left(k+1+N, j\right), \quad 1 \leq k \leq N-1,
\end{aligned}$$
where $\eta$ is a positive parameter and $P(m,k)$ denotes the $k$-permutations of $m$, i.e. the number of permutations of $k$ elements arranged in a specific order in a set of $m$ elements: 
$$P(m,k):= \frac{m!}{(m-k)!} = m\cdot(m-1) \cdots (m-k+1).$$ 
\end{remark}

We now set
$$
\bar{\theta}_k :=\overline{\mathcal{L}}_k^{N}\left(\bar{C}_{k, N}\right), \quad \underline{\theta}_k :=\underline{\mathcal{L}}_1^k\left(\underline{C}_{1, k}\right), \quad 1 \leq k \leq N,
$$
and define a non-negative $q_\theta(t)$ quantity which will be used to control the phase diameter:
\[
q_\theta:=\bar{\theta}_1 - \underline{\theta}_{N}.
\]
By using those newly defined notations, we state two lemmas on a monotone property of the interaction term and a relation between $q_\theta(t)$ and $d_\theta(t)$ whose proofs can be found in \cite[Lemma 4.1]{ZZ23} and \cite[Lemma 4.2]{ZZ23}, respectively. Here, we notice that these lemmas depend only on the graph structure.

\begin{lemma}\label{Min_index}
Consider the strongly connected network $\mathcal{G}$, with phases well-ordered according to $\mathcal{A}$. Moreover, we assume that the phase diameter and free parameter $\eta$ satisfy
\bq\label{condi_eta}
d_\theta(t)<\zeta < \xi<\pi \quad \mbox{and} \quad \eta>\max \left\{\frac{1}{\sin \xi},\frac{1}{\cos{(R_\omega \tau)}}, \frac{2}{1-\frac{\zeta}{\xi}}\right\},
\eq
respectively. Then, we have
\[
\sum_{i=n}^{N}\left(\eta^{i-n} \min\limits_{\substack{j \in \mathcal{N}_i  \\
j \leq i}} \sin \left(\theta_j -\theta_i \right)\right) \leq \sin \left(\theta_{\bar{k}_n} -\theta_{N} \right)
\]
and
\[
\sum_{i=1}^n\left(\eta^{n-i} \max _{\substack{j \in \mathcal{N}_i   \\
j \geq i}} \sin \left(\theta_j^0-\theta_i^0\right)\right) \geq \sin \left(\theta_{\underline{k}_n} -\theta_1 \right),
\]
where 
\[
\bar{k}_n :=\min _{\substack{j \in \cup_{i=n}^{N} \mathcal{N}_i }} j \quad \mbox{and} \quad \underline{k}_n :=\max _{j \in \cup_{i=1}^n \mathcal{N}_i } j \quad \mbox{for $1 \leq n \leq N$}.
\]
\end{lemma}

\begin{lemma}\label{Q_dynamic_strongly_connected}
Consider the strongly connected network $\mathcal{G}$, with the coefficients of the convex combinations satisfying the conditions according to algorithm $\mathcal{A}$. Then we have
$$
\beta d_\theta(t) \leq q_\theta(t) \leq d_\theta(t) \mbox{ with } \beta=1-\frac{2}{\eta},
$$
where $\eta$ satisfies \eqref{condi_eta}.
\end{lemma}

We next provide a result concerning the dynamics of $q_\theta(t)$ which will serve as a cornerstone in deriving our phase bound.

\begin{lemma}
\label{Main_Strong}
\textit{Let $\{\theta_i(t)\}_{i=1}^N$ be a solution to the Kuramoto model \eqref{main_eq2} on a strongly connected digraph $\mathcal{G}$, with initial conditions satisfying 
\[
D_\theta(0)<\pi. 
\]
Let $\zeta$ and $\xi$ such that $D_\theta(0) < \zeta < \xi < \pi$ and let $d_\infty, \eta$ be two parameters such that
$$d_\infty < \min\left\{\frac{\pi}{2},d_\theta(0)\right \}, \quad \quad \eta>\max \left\{\frac{1}{\sin \xi},\frac{1}{\cos{(R_\omega \tau)}}, \frac{2}{1-\frac{\zeta}{\xi}}\right\},$$
where $R_\omega$ is the bound on the velocities defined in \eqref{erre}.
Assume the following conditions hold:}
$$
\begin{aligned}
& \tan (R_\omega \tau)<\frac{ \beta d_\infty}{\left(1+\frac{\zeta}{\zeta-D_\theta(0)}\right) 2 (N-1) c},\\
& d_\infty+R_\omega \tau<\frac{\pi}{2}, \quad \kappa>\left(1+\frac{\zeta}{\zeta-D_\theta(0)}\right) \frac{\left(D(\Omega)+2  \kappa \sin (R_\omega \tau)\right)(N-1) c}{2\cos (R_\omega \tau)} \frac{1}{\beta d_\infty},
\end{aligned}
$$
where 
\[
c:=\frac{\left(\sum_{j=1}^{N-1} \eta^j P\left(2 N, j\right)+1\right) \xi}{\sin \xi}. 
\]
Then we have
\[
d_\theta(t) < \xi, \quad \forall\, t \in [0,\infty)
\]
and
\[
\dot q_\theta(t)  \leq D(\Omega)+2  \kappa \sin (R_\omega \tau)-\frac{2\kappa \cos (R_\omega \tau)}{(N-1)c} q_\theta(t), \quad a.e. \ t \in[0,+\infty). 
\]
\end{lemma}
\begin{remark}\label{rmk_main_strong} From Lemma \ref{Main_Strong}, by applying the Gr\"onwall's lemma, we find
\begin{equation}\label{rev2}
q_\theta(t)  \leq q_\theta(0)\exp\lt( -\frac{ 2\kappa \cos (R_\omega \tau)}{(N-1)c} t\rt)  + \frac{(D(\Omega)+2   \kappa \sin (R_\omega \tau))(N-1)c}{2 \kappa \cos (R_\omega \tau)} \lt(1 - \exp\lt( -\frac{2\kappa \cos (R_\omega \tau)}{(N-1)c} t\rt) \rt).
\end{equation}
In particular, we obtain
\[
q_\theta(t) \leq \max\lt\{d_\theta(0),  \frac{(D(\Omega)+2   \kappa \sin (R_\omega \tau))(N-1)c}{2 \kappa \cos (R_\omega \tau)}\rt\}, \quad \forall \, t\in [0,\infty).
\]
\end{remark}

\begin{proof}[Proof of Lemma \ref{Main_Strong}] 
Let us denote $\Omega_M$ and $\Omega_m$ the maximal and minimal natural frequency, respectively, i.e.
\[
\Omega_M :=\max_{i=1, \dots, N}\, \Omega_i \quad \mbox{and} \quad \Omega_m :=\min_{i=1, \dots, N}\, \Omega_i.
\]
We proceed with the proof in three steps.  

\medskip

\paragraph{\bf Step I} We first define a set $\mathcal{S} : = \left\{T>0: d_\theta(t) < \xi, \ \forall \, t \in[0, T)\right\}$. Since $D_\theta(0) < \xi$ and $d_\theta(t)$ is continuous, the set $\mathcal{S}$ is non-empty. Thus, we can set $T^* := \sup \mathcal{S}$. We claim that $T^* = +\infty$.

Suppose not, i.e. $T^* < + \infty$, then by the continuity of $d_\theta(t)$ we get
\[
d_\theta(t)<\xi, \quad \forall t \in\left[0, T^*\right), \quad d_\theta(T^*)=\xi.
\]
Next, we divide the time interval into sub-intervals in which no two oscillators overlap,
\[
\left[0, T^*\right)=\bigcup_{l} J_l, \quad J_l=\left[t_{l-1}, t_l\right),
\]
where the end-points $t_l$ are the times at which such an overlap occurs. 
We then apply the well-ordering of phases in each interval:
\[
\theta_1 (t) \leq \theta_2 (t) \leq \ldots \leq \theta_{N} (t), \quad t \in J_l .
\]

\medskip

\paragraph{\bf Step II} We claim that for $1 \leq n \leq N$
\bq\label{claim_main}
\dot{\bar{\theta}}_n (t) \leq \Omega_M + \kappa \sin (R_\omega \tau) +  \frac{\kappa \cos (R_\omega \tau)}{(N-1)(\bar{a}_n +1)} \sum_{i=n}^{N} \eta^{i-n} \min _{\substack{j \in \mathcal{N}_i  \\ j \leq i}} \sin \left(\theta_j (t)-\theta_i (t)\right).
\eq
For this, we use the inductive argument. Note that $\bar{\theta}_{N} =\theta_{N} $, and thus by mean-value theorem, we obtain
\[\begin{aligned}
\dot{\theta}_{N}(t) & =\Omega_{N}+ \frac{\kappa}{N -1} \sum_{j \in \mathcal{N}_{N}} \sin \left(\theta_j(t-\tau_{jN})-\theta_{N}(t)\right) \\
& =\Omega_{N} + \frac{\kappa}{N-1} \sum_{j \in \mathcal{N}_{N}} \sin \left(\theta_j(t)-\theta_{N}(t) -\dot{\theta}_{j}(t_{jN}^*)\tau_{jN}\right) \\
& \leq \Omega_M+ \frac{\kappa}{N-1} \sum_{j \in \mathcal{N}_{N}}\left[\sin \left(\theta_j(t) -\theta_{N}(t)\right) \cos (\dot{\theta}_{j}(t_{jN}^*)\tau_{jN}) -\cos \left(\theta_j(t)-\theta_{N}(t)\right) \sin (\dot{\theta}_{j}(t_{jN}^*)\tau_{jN})\right] \\
& \leq \Omega_M +  \kappa \sin (R_\omega \tau) +  \frac\kappa{N-1} \cos (R_\omega \tau)  \min _{j \in \mathcal{N}_{N}} \sin \left(\theta_j(t)-\theta_{N}(t)\right) 
\end{aligned}\]
for some $t_{jN}^* \in (t-\tau_{jN},t)$ and all $t \in [0,T^*)$. This shows that \eqref{claim_main} holds for $n = N$. Now we assume that  \eqref{claim_main} holds for $n \in [2,N-1]$. Note that 
\begin{align*}
& \frac\kappa{N-1}\sum_{j \in \mathcal{N}_{n-1}} \sin (\theta_j(t-\tau_{j (n-1)}) -\theta_{n-1}(t)) \cr
&\quad \leq  \kappa \sin (R_\omega \tau) +  \frac\kappa{N-1} \lt(\sum_{\substack{j \in \mathcal{N}_{n-1} \\ j\leq n-1}} + \sum_{\substack{j \in \mathcal{N}_{n-1} \\ j>n-1}}  \rt)\sin (\theta_j(t) -\theta_{n-1}(t)) \cos (\dot{\theta}_{j}(t_{j(n-1)}^*)\tau_{j(n-1)})\cr
&\quad \leq  \kappa \sin (R_\omega \tau) +  \frac\kappa{N-1}\cos (R_\omega \tau) \min_{\substack{j \in \mathcal{N}_{n-1} \\ j\leq n-1}}\sin (\theta_j(t) -\theta_{n-1}(t)) + \frac\kappa{N-1} \sum_{\substack{j \in \mathcal{N}_{n-1} \\ j>n-1}}   \sin (\theta_j(t) -\theta_{n-1}(t)).
\end{align*}
By using the above together with 
\[
\frac{\bar a_{n-1}}{\bar a_n + 1} = \eta(2N-n+2) = \eta N + \eta + \eta (N - n + 1),
\]
we estimate
\begin{align*}
 \dot{\bar{\theta}}_{n-1} &= \frac{d}{d t}\left(\frac{\bar{a}_{n-1} \bar\theta_{n}+\theta_{n-1}}{\bar{a}_{n-1}+1}\right)=\frac{\bar{a}_{n-1}}{\bar{a}_{n-1}+1} \dot{\bar \theta}_{n}+\frac{1}{\bar{a}_{n-1}+1} \dot{\theta}_{n-1}   \cr
 &\leq \frac{\bar{a}_{n-1}}{\bar{a}_{n-1}+1} \lt( \Omega_M +   \kappa \sin (R_\omega \tau)+ \frac{\kappa \cos (R_\omega \tau)}{(N-1)(\bar{a}_n +1)}  \sum_{i=n}^{N} \eta^{i-n} \min _{\substack{j \in \mathcal{N}_i \\ j \leq i}} \sin \left(\theta_j(t)-\theta_i(t)\right) \rt)\cr
 &\quad + \frac{1}{\bar{a}_{n-1}+1}\lt( \Omega_M + \frac\kappa{N-1} \sum_{j \in \mathcal{N}_{n-1}} \sin (\theta_j(t-\tau_{j(n-1)}) -\theta_{n-1}(t)) \rt) \cr
 &\leq \Omega_M +  \kappa \sin (R_\omega \tau) + \frac{\kappa \cos (R_\omega \tau) N }{(N-1)(\bar a_{n-1}+1)} \sum_{i=n}^{N} \eta^{i-n+1} \min _{\substack{j \in \mathcal{N}_i \\ j \leq i}} \sin \left(\theta_j(t)-\theta_i(t)\right) \\ 
& + \frac{\kappa \cos (R_\omega \tau)}{(N-1)(\bar a_{n-1}+1)} \lt( \sum_{i=n}^{N} \eta^{i-n+1} \min _{\substack{j \in \mathcal{N}_i \\ j \leq i}} \sin \left(\theta_j(t)-\theta_i(t)\right) +    \min_{\substack{j \in \mathcal{N}_{n-1} \\ j\leq n-1}}\sin (\theta_j(t) -\theta_{n-1}(t))\rt) \\ 
& +\frac{\kappa}{(N-1)(\bar a_{n-1}+1)}  \left(\eta\left(N-n+1\right)\cos(R_\omega \tau) \sum_{i=n}^{N} \eta^{i-n} \min _{\substack{j \in \mathcal{N}_i  \\ j \leq i}} \sin \left(\theta_j-\theta_i\right)+\sum_{\substack{j \in \mathcal{N}_{n-1}  \\ j>n-1}} \sin \left(\theta_j-\theta_{n-1}\right)\right).
\end{align*}
Here the third term on the right-hand side is nonpositive, and the fourth term can be written as 
\[
\frac{\kappa \cos (R_\omega \tau)}{(N-1)(\bar a_{n-1}+1)} \lt( \sum_{i=n-1}^{N} \eta^{i-(n-1)} \min _{\substack{j \in \mathcal{N}_i \\ j \leq i}} \sin \left(\theta_j(t)-\theta_i(t)\right)  \rt).
\]
Thus, to prove \eqref{claim_main}, it suffices to show that the last term is nonpositive. 

We first observe from Lemma \ref{Min_index} that
\[
\sum_{i=n}^{N} \eta^{i-n} \min _{\substack{j \in \mathcal{N}_i \\ j \leq i}} \sin \left(\theta_j-\theta_i\right) \leq \sin \left(\theta_{\bar{k}_n}-\theta_{N}\right), \quad \bar{k}_n :=\min _{\substack{j \in \cup_{i=n}^{N} \mathcal{N}_i}} j. 
\]
We then consider the case $\xi > \frac{\pi}{2}$. If $\theta_{N} - \theta_{\bar{k}_{n}} \leq \frac{\pi}{2}$, then by noticing $\bar k_n \leq n-1$ and $\eta \cos(R_\omega \tau)  >1$, we find
\begin{align*}
&\eta\left(N-n+1\right)\cos(R_\omega \tau) \sum_{i=n}^{N} \eta^{i-n} \min _{\substack{j \in \mathcal{N}_i \\ j \leq i}} \sin \left(\theta_j-\theta_i\right)+\sum_{\substack{j \in \mathcal{N}_{n-1} \\ j>n-1}} \sin \left(\theta_j-\theta_{n-1}\right)\cr
&\quad \leq \eta\left(N-n+1\right)\cos(R_\omega \tau)  \sin \left(\theta_{\bar{k}_n}-\theta_{N}\right) + \left(N-n+1\right) \sin (\theta_N - \theta_{n-1})\cr
&\quad \leq 0.
\end{align*}
On the other hand, if $\frac{\pi}{2}<\theta_{N} - \theta_{\bar{k}_{N}} < \xi$, we use 
\[
\eta>\frac{1}{\sin \xi} \quad  \text {and} \quad  \sin \left(\theta_{N} -\theta_{\bar{k}_{n}} \right)>\sin \xi
\]
to see $\eta \sin \left(\theta_{\bar{k}_{N}}-\theta_{N}\right) \leq-1$, and this gives the nonpositivity of the last term. The case for $\xi \leq \frac{\pi}{2}$ follows similarly. 

We now use \eqref{claim_main} and Lemma \ref{Min_index} for $n=1$ and find that 
\begin{align*} \dot{\bar{\theta}}_1 & \leq \Omega_M+  \kappa   \sin (R_\omega \tau)+  \frac{\kappa \cos (R_\omega \tau)}{(N-1)(\bar{a}_1+1)} \sum_{i=1}^{N} \eta^{i-1} \min _{\substack{j \in \mathcal{N}_i \\ j \leq i}} \sin \left(\theta_j-\theta_i\right) \\ 
& \leq \Omega_M+ \kappa \sin (R_\omega \tau)+ \frac{\kappa \cos (R_\omega \tau)}{(N-1)(\bar{a}_1+1)}  \sin \left(\theta_{\bar k_1}-\theta_{N}\right) \\ 
& =\Omega_M+ \kappa \sin (R_\omega \tau)+ \frac{\kappa \cos (R_\omega \tau)}{(N-1)(\bar{a}_1+1)} \sin \left(\theta_1-\theta_{N}\right),
\end{align*}
where we used the strong connectivity of $\mathcal{G}$, and which completes the process for $\mathcal{A}_1$. 

\medskip

\paragraph{\bf Step III} We can similarly build from bottom to top with $\mathcal{A}_2$ to arrive at
\begin{align*} 
\dot{\underline{\theta}}_{N}(t) & \geq \Omega_m - \kappa   \sin (R_\omega \tau)+  \frac{\kappa \cos (R_\omega \tau)}{(N-1)(\underline{a}_N+1)} \sum_{i=1}^{N} \eta^{i-1} \min _{\substack{j \in \mathcal{N}_i \\ j \geq i}} \sin \left(\theta_j-\theta_i\right) \\ 
& \geq \Omega_m -   \kappa \sin (R_\omega \tau)+\frac{\kappa \cos (R_\omega \tau)}{(N-1)(\underline{a}_N+1)} \sin \left(\theta_{\underline{k}_N}-\theta_{1}\right) \\ 
& =\Omega_ m -   \kappa \sin (R_\omega \tau)+  \frac{\kappa \cos (R_\omega \tau)}{(N-1)(\bar{a}_1 +1)} \sin \left(\theta_N-\theta_1\right),
\end{align*}
from which we obtain
$$
\begin{aligned} 
\dot{q}_\theta(t) &  \leq D(\Omega)+2  \kappa \sin (R_\omega \tau) - \frac{2\kappa \cos (R_\omega \tau)}{(N-1)(\bar{a}_1 +1)} \sin \left(\theta_{N}-\theta_1\right) \\ 
& \leq D(\Omega)+2  \kappa \sin (R_\omega \tau)- \frac{2\kappa \cos (R_\omega \tau)}{N-1}  \frac{1}{\sum_{j=1}^{N-1} \eta^j P\left(2 N, j\right)+1} \sin \left(\theta_{N}-\theta_1\right),\end{aligned}
$$
where we used
$$
\bar{a}_1=\sum_{j=1}^{N-1} \eta^j P\left(2 N, j\right) .
$$
Since the function $\frac{\sin x}{x}$ is monotonically decreasing in $(0, \pi]$,  we obtain 
$$
\sin \left(\theta_{N}-\theta_1\right) \geq \frac{\sin \xi}{\xi}\left(\theta_{N}-\theta_1\right).
$$
Moreover, since $q_\theta(t) \leq \theta_{N}(t)-\theta_1(t)$, 
\begin{align}\label{ineq_Q0}
\begin{aligned}
\dot{q}_\theta(t) 
& \leq D(\Omega)+2  \kappa \sin (R_\omega \tau) - \frac{ 2\kappa \cos (R_\omega \tau) }{(N-1)c} q(t)
\end{aligned}
\end{align}
for a.e. $t \in(0,T^*)$. Then, by Lemma \ref{Q_dynamic_strongly_connected} and Remark \ref{rmk_main_strong}, we further have
\[
\beta d_\theta(t) \leq q_\theta(t)  \leq \max\lt\{D_\theta(0),  \frac{(D(\Omega)+2   \kappa \sin (R_\omega \tau))(N-1)c}{2 \kappa \cos (R_\omega \tau)}\rt\}. 
\]
On the other hand, by the hypotheses, we get
$$
\frac{D_\theta(0)}{\beta} < \frac\zeta\beta = \frac\zeta{1 - \frac2\eta} < \xi \quad \mbox{and} \quad \frac{(D(\Omega)+2  \kappa \sin (R_\omega \tau))(N-1)c}{ 2\beta \kappa \cos (R_\omega \tau)} < d_\infty < \zeta < \xi,
$$
and hence,
$$
\xi =  d_\theta(T^*) \leq \frac1\beta\max\left\{D_\theta(0),  \frac{(D(\Omega)+2  \kappa \sin (R_\omega \tau))(N-1)c}{ 2\kappa \cos (R_\omega \tau)}\right\} < \xi.
$$
This is a contradiction, and thus $T^*=\infty$, i.e. $d_\theta(t) < \xi$ for all $[0,\infty)$. Moreover, \eqref{ineq_Q0} holds for a.e. $t \in (0,\infty)$.
\end{proof}

\begin{lemma}
Let $\{\theta_i(t)\}_{i=1}^N$ be a solution to the Kuramoto model \eqref{main_eq2} on a strongly connected digraph $\mathcal{G}$ and assume the hypothesis of Lemma \ref{Main_Strong} holds. Then, there exists a finite time $t_*$ such that 
\begin{equation}\label{diam_phase_small}
d_\theta(t) \leq d_\infty \quad \text { for } t \in\left[t_*,+\infty\right).
\end{equation}
\end{lemma}
\begin{proof}
    We first observe that
\begin{equation}
\label{main_obser}
\frac{(D(\Omega)+2  \kappa \sin (R_\omega \tau))(N-1)c}{2\kappa \cos (R_\omega \tau)} < \beta d_\infty < \beta d_\theta(0) \leq q_\theta(0).
\end{equation}\
Now, we set 
\begin{align*}
f(t) :=
 q_\theta(0) \exp\left(-\frac{2\kappa \cos(R_\omega \tau)}{(N-1)c} t\right)  + \frac{(D(\Omega)+2  \kappa \sin(R_\omega \tau))(N-1)c}{2\kappa \cos(R_\omega \tau)} \left(1 - \exp\left(-\frac{2\kappa \cos(R_\omega \tau)}{(N-1)c} t\right)\right).
\end{align*}
Then $f(t)$ is continuous on $[0,\infty)$, monotonic, and 
\begin{align*}
\lim_{t \to \infty} f(t) &
= \frac{(D(\Omega)+2  \kappa \sin (R_\omega \tau))(N-1)c}{ 2\kappa \cos (R_\omega \tau)}. 
\end{align*}
This together with \eqref{main_obser} yields that there exists $t_* > 0$ such that
\[
f(t) \leq \beta d_\infty, \quad \forall \, t \in [t_*,\infty).
\]
Then, we now use Lemma \ref{Q_dynamic_strongly_connected}, Lemma \ref{Main_Strong}, and our Gr\"onwall inequality \eqref{rev2} to conclude that
\[
d_\theta(t) \leq \frac{1}{\beta} f(t) \leq d_\infty, \quad \forall \, t \in [t_*,\infty).
\]
This completes the proof.
\end{proof}

%
%
%
%
%
%
%
%
%

\section{Asymptotic synchronization: strongly connected case} \label{sec_syn_str}
In this section, we discuss the asymptotic frequency synchronization result for solutions to system \eqref{main_eq2}, i.e. the frequency-diameter decays to zero as time tends to infinity. For this, by differentiating the system \eqref{main_eq2} with respect to $t$, we obtain
\begin{align}\label{second_order}
\begin{aligned}
\frac{d}{dt}\theta_i(t) &= \omega_i(t), \quad i=1,\dots, N, \quad t > 0,\cr
\frac{d}{dt} \omega_i(t) &= \frac\kappa {N-1}\sum_{k\ne i} \chi_{ij}\cos(\theta_k(t-\tau_{ik}) - \theta_i(t)) (\omega_k(t-\tau_{ik}) - \omega_i(t)).
\end{aligned}
\end{align}
In order to prove the diameter decay estimate for solutions to the above second-order system, we need some preliminary results.

In the rest of this section, we assume that the hypotheses of Lemma \ref{Main_Strong} are satisfied. Then, \eqref{diam_phase_small} holds for $t\ge t_*>0.$ Without loss of generality, we may assume $t_*>\tau$, and we deduce the following estimates.
\begin{remark} \label{rem1}
If  $\{\theta_i(t)\}_{i=1}^N$  is a global solution of  \eqref{main_eq}, then 
\begin{equation*}
\begin{split}
|\theta_i(t-\tau_{ik})-\theta_j(t)| & \leq |\theta_i(t-\tau_{ik})-\theta_i(t)|+|\theta_i(t)-\theta_j(t)| \\
& \leq R_\omega \tau+d_\theta(t) \\
& \leq R_\omega \tau + d_\infty < \frac{\pi}{2}, \ \forall\, t\ge t_*.
\end{split}
\end{equation*}
Thus, if we denote $\xi_*=\cos \Big( R_\omega \tau + d_\infty \Big),$ we have that
\begin{equation} \label{lower bound}
\cos (\theta_i(t - \tau_{ik})-\theta_j(t)) \geq  \xi_* >0, \ \forall\, t \geq t_*.
\end{equation}
\end{remark}
Let us denote
\[
 \tau_i := \max_{i\in \mathcal{N}_j} \{\tau_{ji} \} \quad \mbox{and} \quad \tau_0 := \min_{i=1,\dots,N} \{ \tau_i \}.
 \]

\begin{definition}\label{m0eM0}
We define two numbers, $M_0$ and $m_0,$ as 
\begin{equation} \label{Azero}
\begin{split}
M_0:= \max_{i=1,\dots, N} \, \max_{t \in I_0^i}\  \omega_i(t) \quad \mbox{and} \quad
 m_0:= \min_{i=1,\dots,N} \, \min_{t \in I_0^i} \ \omega_i(t),
\end{split}
\end{equation}
respectively, where $I_0^i:= [t_*- \tau_i, t_*].$
\end{definition}

First, we show that the velocities remain bounded. Indeed, we have the following lemma.
\begin{lemma} \label{5.1}
Let $\{\theta_i(t)\}_{i=1}^N$ a global classical solution of  \eqref{main_eq2}. Then, for all $i=1, \dots, N,$ we have 
\begin{equation}\label{bounded_velocities}
m_0 \leq \omega_i(t) \leq M_0, 
\end{equation}
for all $t \geq t_*-\tau_i,$ with $m_0$ and $M_0$ given by \eqref{Azero}.
\end{lemma}

\begin{proof} 
Fix
$\epsilon > 0$ and let us set 
$$ {\mathcal T}^{\epsilon}:= \Big \{ t > t_*\  : \  \max_{i=1,...,N} \omega_i(s) < M_0+\epsilon \quad \forall s \in [t_*,t) \Big \}. $$
Since the inequality \eqref{bounded_velocities} is trivial in $[t_*-\tau_i,t_*],$ $i=1, \dots, N,$ by continuity we deduce that ${\mathcal T}^{\epsilon} \neq \emptyset.$ Let us denote ${\mathcal S}^{\epsilon}:= \sup {\mathcal T}^{\epsilon}$. By definition of ${\mathcal T}^{\epsilon}$, trivially ${\mathcal S}^{\epsilon}>0$. We claim that ${\mathcal S}^{\epsilon} = +\infty $. Assume by contradiction that ${\mathcal S}^{\epsilon}< +\infty$. By definition of $ {\mathcal S}^{\epsilon}$ we have that 
\begin{equation*}
\max_{i=1,...,N} \omega_i(t) < M_0 + \epsilon, \quad \forall t \in [t_*, {\mathcal S}^{\epsilon}), 
\end{equation*}
and 
\begin{equation*}
\lim_{t \rightarrow {\mathcal S}^{\epsilon-}} \max_{i=1,...,N} \omega_i(t) = M_0 + \epsilon\,. 
\end{equation*}
For every $i=1,...,N $, $ \forall t \in (t_*, {\mathcal S}^{\epsilon}),$ we obtain
\begin{equation*}
\frac{d}{d t}\omega_i(t)= \frac{\kappa}{N-1} \sum_{k \neq i} \cos(\theta_k(t-\tau_{ik})-\theta_i(t)) ( \omega_k(t-\tau_{ik})-\omega_i(t)). 
\end{equation*}
Note that, from \eqref{lower bound} $\cos(\theta_k(t-\tau_{ik})-\theta_i(t))>0,$ for all $i,k=1, \dots,N,$ and for every $t\ge t_*.$
Moreover, if $t \in (t_*, {\mathcal S}^{\epsilon})$, then $t-\tau_{ik} \in (t_*-\tau, {\mathcal S}^{\epsilon})$ and thus,
$$ \omega_k(t-\tau_{ik})< M_0 + \epsilon, \quad \forall k=1,...,N.$$ 
Then, it follows from \eqref{second_order} that
\begin{align*}
\begin{aligned}
\frac{d}{d t} \omega_i(t)  \leq \frac{\kappa}{N-1} \sum_{k \neq i} \cos(\theta_k(t-\tau_{ik})-\theta_i(t))(M_0+\epsilon -  \omega_i(t)) \leq \kappa (M_0+ \epsilon -  \omega_i(t)),
\end{aligned}
\end{align*}
where we used that also  
$\omega_i(t) < M_0 + \epsilon$, for all $i=1,...,N,$
and thus, $M_0 + \epsilon -  \omega_i(t) \geq 0$.

Using the Gr\"onwall's lemma, we find that 
\begin{align*}
\begin{aligned}  \omega_i(t)& \leq e^{-\kappa (t-t_*)}\omega_i(t_*) + \kappa (M_0+\epsilon) \int_{t_*}^{t} e ^{-\kappa (t-s)} \, ds \cr
& = e^{-\kappa (t-t_*)}  \omega_i(t_*)+ (M_0+\epsilon)(1-e^{-\kappa (t-t_*)}) \cr
& \leq e^{-\kappa (t-t_*)}M_0+M_0+\epsilon - e^{-\kappa (t-t_*)}M_0 - \epsilon e^{-\kappa (t-t_*)} \cr
& = M_0 + \epsilon - \epsilon e^{-\kappa (t-t_*)} \leq M_0 + \epsilon - \epsilon e^{-\kappa ({\mathcal S}^{\epsilon}-t_*)}\,.
\end{aligned}
\end{align*}
This implies
$$ \max_{i=1,...,N} \omega_i(t) \leq M_0 + \epsilon -\epsilon e^{-\kappa (S^{\epsilon}-t_*)}, \quad \forall\, t \in (t_*, {\mathcal S}^{\epsilon}). $$
Taking the limit for $ t \rightarrow {\mathcal S}^{\epsilon-}$, we have that 
$$ \lim_{t\rightarrow {\mathcal S}^{\epsilon-}}  \max_{i=1,...,N} \omega_i(t) \leq M_0 + \epsilon - \epsilon e^{-\kappa ({\mathcal S}^{\epsilon}-t_*)} < M_0 + \epsilon, $$
and this gives a contradiction. Then, ${\mathcal S}^{\epsilon} = +\infty$, and subsequently, we arrive at
$$ \max_{i=1,...,N} \omega_i(t)< M_0 + \epsilon, \quad \forall \,t \ge t_*. $$
Since $\epsilon > 0$ is arbitrary, we conclude that 
$$ \max_{i=1,...,N}\omega_i(t)\leq M_0, $$ and  then
$$ \omega_i(t)\leq M_0, \quad \forall \, t \geq t_*-\tau_i,\  \forall \, i=1,...,N.$$
Applying the above argument to $-\omega_i(t),$ $t \geq t_*,$ we have
\begin{align*}
\begin{aligned}
 - \omega_i(t)& \leq \max_{j=1,...,N} \max_{s \in [t_*-\tau_i, t_*]} \{- \omega_j(s)\} = - \min_{j=1,...,N} \min_{s \in [t_*-\tau_i, t_*]}\omega_j(s) = -m_0.
\end{aligned}
\end{align*}
Hence we deduce
$$ \omega_i(t) \geq m_0,\quad \forall\, t \geq t_*-\tau_i,\  \forall\, i=1,...,N. $$
This concludes the proof.
\end{proof}

Without loss of generality, being system \eqref{main_eq2} invariant by translation, we may assume
\begin{equation*}
0<m_0\le M_0.
\end{equation*}
 
\begin{definition}\label{mneMn}
For all $n \in \mathbb{N}$ we define the quantities $M_n$ and $m_n$ as 
\begin{equation} \label{A}
\begin{split}
M_n:= \max_{i=1,\dots, N} \,\max_{t \in I_n^i} \ \omega_i(t) \quad \mbox{and} \quad 
m_n:= \min_{i=1,\dots,N} \,\min_{t \in I_n^i} \ \omega_i(t),
\end{split}
\end{equation}
respectively, where $I_n^i:=[2\gamma n\tau +t_*- \tau_i, 2\gamma n\tau+t_*].$
\end{definition}

Note that, for $n=0,$ from \eqref{A} we obtain the constants already defined in \eqref{Azero}.

\begin{remark} \label{remark5.1}{
Arguing as in  Lemma \ref{5.1}
we have that
\begin{equation*}
m_n \leq \omega_i(t) \leq M_n, 
\end{equation*}
for all $n \in \mathbb{N}$ and $t \geq 2\gamma n \tau +t_*-\tau_i.$}
\end{remark}

Now, recalling \eqref{erre}, we  define the  quantities: 

\begin{equation} \label{sigma}
\sigma:= \min \Big \{\tau_0, \frac{M_0-m_0}{4\kappa R_\omega}\Big \},
\end{equation}
and
\begin{equation} \label{Gamma}
\Gamma:=\Big(\frac{\xi_*}{N-1}\Big)^{\gamma}e^{-\kappa2\gamma\tau}(1-e^{-\frac{\kappa \tau}{N-1}})^{\gamma-1}(1-e^{-\frac{\kappa}{N-1}\sigma}).
\end{equation}

The following lemma extends to the network structure an argument of \cite{Has21}.
\begin{lemma}\label{lemma 3}
Let $\{\theta_i(t)\}_{i=1}^N$ a global classical solution of  \eqref{main_eq2}. Then,
\begin{equation*}
m_0+\frac{\Gamma}{2}(M_0-m_0) \leq \omega_i(t)\leq M_0-\frac{\Gamma}{2}(M_0-m_0), \quad t \in [t_*+(2\gamma-1)\tau, t_*+2\gamma\tau],
\end{equation*}
for all $i \in \{ 1, \dots, N \},$  where $\Gamma$ is defined by \eqref{Gamma}.
\end{lemma}
\begin{proof}
Let $L \in \{1,\dots,N\}$ such that $\omega_L(s)= m_0 $ for some $s \in [t_*-\tau_L, t_*].$ Since it is true that $\omega_L(t) \leq R_\omega$, then we have that $|\dot{\omega}_L(t)| \leq 2 \kappa R_{\omega}$. Therefore, one can find a closed interval $[\alpha_L,\beta_L] \subset [t_*-\tau_L, t_*]$ of length $\sigma$, defined as in \eqref{sigma}, such that
$$ m_0 \leq  \omega_L(t) \leq \frac{M_0+m_0}{2}, \ t \in [\alpha_L,\beta_L]. $$
Let $i_1 \in \{1,\dots,N\}\setminus \{L\}$ such that $\chi_{i_1L}=1$ and $\tau_{i_1L}=\tau_L$. Consider $t \in [\alpha_L+\tau_L, \beta_L+\tau_L].$ From the equation \eqref{second_order} we have: 
\begin{equation*}
\begin{split}
\frac{d}{d t}  \omega_{i_1}(t) &= \frac{\kappa}{N-1} \sum_{\substack{j \neq i_1 \\ j \neq L}} \chi_{i_1j} \cos(\theta_j(t-\tau_{i_1j})-\theta_{i_1}(t)) ( \omega_j(t-\tau_{i_1j})-\omega_{i_1}(t)) \\
&\quad + \frac{\kappa}{N-1}\cos(\theta_L(t-\tau_{i_1L})-\theta_{i_1}(t)) (\omega_L(t-\tau_{i_1L})-\omega_{i_1}(t)).
\end{split}
\end{equation*}
Notice that 
$$ \sum_{j \neq i} \chi_{ij} \cos(\theta_j(t-\tau_{ij})-\theta_i(t)) \leq \sum_{j \neq i} \chi_{ij} = N_i. $$
Then, using Remark \ref{rem1}, for $ t\in [\alpha_L+\tau_L,\beta_L+\tau_L],$ we estimate
\begin{equation*}
\begin{split}
& \frac{d}{d t} \omega_{i_1}(t) \leq \frac{\kappa}{N-1}\sum_{\substack{j \neq i_1 \\ j \neq L}} \chi_{i_1j} \cos(\theta_j(t-\tau_{i_1j})-\theta_{i_1}(t))(M_0-\omega_{i_1}(t)) \\
& \hspace{1,5 cm}+ \frac{\kappa}{N-1}\cos (\theta_L(t-\tau_{i_1L})-\theta_{i_1}(t))\Big(\frac{M_0+m_0}{2}-\omega_{i_1}(t)\Big) \\
& \hspace{1 cm}\leq \frac{\kappa}{N-1} (N_{i_1}-\cos(\theta_L(t-\tau_{i_1L})-\theta_{i_1}(t)))(M_0- \omega_{i_1}(t)) \\
& \hspace{1,5 cm}+ \frac{\kappa}{N-1}\cos (\theta_L(t-\tau_{i_1L})-\theta_{i_1}(t))\Big(\frac{M_0+m_0}{2}- \omega_{i_1}(t)\Big) \\
& \hspace{1 cm}= \frac{\kappa}{N-1} (N_{i_1}-\cos(\theta_L(t-\tau_{i_1L})-\theta_{i_1}(t)))(M_0- \omega_{i_1}(t)) \\
& \hspace{1,5 cm}+ \frac{\kappa}{N-1}\cos (\theta_L(t-\tau_{i_1L})-\theta_{i_1}(t))\Big(\frac{M_0+m_0}{2}- \omega_{i_1}(t)\Big) \pm \frac{\kappa}{N-1}\cos(\theta_L(t-\tau_{i_1L})-\theta_{i_1}(t))M_0 \\
& \hspace{1 cm}= \frac{\kappa}{N-1}N_{i_1}M_0-\cos(\theta_L(t-\tau_{i_1L})-\theta_{i_1}(t))\frac{M_0-m_0}{2} -\frac{\kappa}{N-1}N_{i_1}  \omega_{i_1}(t) \\
& \hspace{1 cm}\leq \frac{\kappa}{N-1}(N_{i_1}M_0-\xi_*\frac{M_0-m_0}{2})-\frac{\kappa}{N-1}N_{i_1}\omega_{i_1}(t).
\end{split}
\end{equation*}
Integrating the above inequality over $t \in [\alpha_L+\tau_L,\beta_L+\tau_{L}]$ gives
$$ \omega_{i_1}(t) \leq e^{-\frac{\kappa}{N-1}N_{i_1}(t-\alpha_L-\tau_{L})} \omega_{i_1}(\alpha_L+\tau_{L}) + \frac{1}{N_{i_1}}\Big( N_{i_1}M_0-\xi_*\frac{M_0-m_0}{2} \Big) (1-e^{-\frac{\kappa}{N-i}N_{i_1}(t-\alpha_L-\tau_{L})}).$$
Putting $t=\beta_L+\tau_{L}$ we find 
\begin{equation}\label{s12} \begin{split} 
\omega_{i_1}(\beta_L+\tau_{L}) & \leq e^{-\frac{\kappa}{N-1}N_{i_1}\sigma} \omega_{i_1}(\alpha_L+\tau_{L}) + \frac{1}{N_{i_1}}\Big(N_{i_1}M_0-\xi_*\frac{M_0-m_0}{2} \Big) (1-e^{-\frac{\kappa}{N-1}N_{i_1}\sigma})\\
& \leq M_0-(1-e^{-\frac{\kappa}{N-1}N_{i_1}\sigma})\xi_*\frac{M_0-m_0}{2N_{i_1}}.
\end{split}
\end{equation}
Denoting
$$ \delta_{-}:=\frac{\xi_*}{2(N-1)}(1-e^{-\frac{\kappa}{N-1}\sigma})\Big(1-\frac{m_0}{M_0} \Big), $$
from \eqref{s12} we have that 
\begin{equation} \label{omega1}
 \omega_{i_1}(\beta_L+\tau_{L}) \leq (1-\delta_{-})M_0.
\end{equation}
Consider now $t \in [\beta_L+\tau_{L}, t_* +2\gamma\tau].$ Note that
\begin{equation*}
\begin{split}
\frac{d}{d t} \omega_{i_1}(t) & = \frac{\kappa}{N-1} \sum_{j \neq i_1} \chi_{i_1j} \cos(\theta_j(t-\tau_{i_1j})-\theta_{i_1}(t)) (\omega_j(t-\tau_{i_1j})-\omega_{i_1}(t))\\
& \leq \frac{\kappa}{N-1}N_{i_1}(M_0- \omega_{i_1}(t)).
\end{split}
\end{equation*} 
Integrating over $[\beta_L+\tau_{L},t]$ with $t \in [\beta_L+\tau_{L}, t_*+2\gamma\tau]$, we obtain 
\begin{equation*}
\begin{split}
 \omega_{i_1}(t) & \leq e^{-\frac{\kappa}{N-1}N_{i_1}(t-\beta_L-\tau_{L})}  \omega_{i_1}(\beta_L+\tau_L) + (1-e^{-\frac{\kappa}{N-1}N_{i_1}(t-\beta_L-\tau_{L})})M_0 \\
& \leq e^{-\frac{\kappa}{N-1}N_{i_1}(t-\beta_L-\tau_{L})} (1-\delta_{-})M_0 + (1-e^{-\frac{\kappa}{N-1}N_{i_1}(t-\beta_L-\tau_{L})})M_0 \\
& \leq M_0(1-e^{-\frac{\kappa}{N-1}N_{i_1}2\gamma\tau}\delta_{-}),
\end{split}
\end{equation*}
where we used \eqref{omega1}.
Using $1 \leq N_{i_1} \leq N-1$, we find 
\begin{equation}\label{s13}
\omega_{i_1}(t) \leq M_0(1-e^{-\kappa2\gamma\tau}\delta_{-}), \ t \in [\beta_L+\tau_{L},t_*+2\gamma\tau].
\end{equation}
Consider now $i_2 \in \{1, \dots, N \} \setminus \{i_1\} $ such that $\chi_{i_2i_1}=1$ and consider $t \in [t_*+2\tau, t_*+2\gamma\tau].$ Again, from equation \eqref{second_order} we have
\begin{equation*}
\begin{split}
\frac{d}{d t}  \omega_{i_2}(t) & = \frac{\kappa}{N-1} \sum_{\substack{j \neq i_2 \\ j \neq i_1}} \chi_{i_2j} \cos(\theta_j(t-\tau_{i_2j})-\theta_{i_2}(t)) (\omega_j(t-\tau_{i_2j})-\omega_{i_2}(t)) \\
& \hspace{0.8 cm}+ \frac{\kappa}{N-1} \cos(\theta_{i_1}(t-\tau_{i_2i_1})-\theta_{i_2}(t))( \omega_{i_1}(t-\tau_{i_2i_1})-\omega_{i_2}(t)) \\
& \leq \frac{\kappa}{N-1}(N_{i_2}-\cos(\theta_{i_1}(t-\tau_{i_2i_1})-\theta_{i_2}(t)))(M_0-\omega_{i_2}(t))\\
&  \hspace{0.8 cm}+\frac{\kappa}{N-1}\cos(\theta_{i_1}(t-\tau_{i_2i_1})-\theta_{i_2}(t))[M_0(1-e^{-\kappa2\gamma\tau}\delta_{-})- \omega_{i_2}(t) ]\\
& \leq \frac{\kappa}{N-1}M_0(N_{i_2}-e^{-\kappa2\gamma\tau}\xi_*\delta_{-})-\frac{\kappa}{N-1}N_{i_2}\omega_{i_2}(t),
\end{split}
\end{equation*}
where we used \eqref{s13}.
Integrating over $[t_*+2\tau,t]$ with $t \in [t_*+2\tau, t_*+2\gamma\tau]$ deduces
\begin{equation*}
\omega_{i_2}(t) \leq M_0 \left [ 1-e^{-\kappa2\gamma\tau}\xi_*\delta_{-}
 \left (\frac{1-e^{-\frac{\kappa}{N-1}(t-t_*-2\tau)}}{N-1} \right )\right ].
\end{equation*}
Then, for $t\in [t_*+3\tau, t_*+2\gamma\tau],$ we have
\begin{equation*}
 \omega_{i_2}(t) \leq M_0 \Big( 1-e^{-\kappa2\gamma\tau}\xi_*\delta_{-} \Big(\frac{1-e^{-\frac {\kappa\tau} {N-1}}}{N-1} \Big)\Big).
\end{equation*}
Iterating this process, along the path starting from $i_1,$  we obtain the following upper bound:
\begin{equation*} 
\omega_{i_n}(t) \leq M_0 \left [ 1-e^{-\kappa2\gamma\tau}\delta_{-} \xi_*^{n-1} \left ( \frac{1-e^{-\frac {\kappa\tau} {N-1}}}{N-1} \right )^{n-1}\right ], 
\end{equation*}
with $n$ such that $2\leq n\leq\gamma$ and $t \in [t_*+(2n-1)\tau, t_*+ 2\gamma\tau].$

Note that, being the digraph strongly connected, from $i_1$ one can reach any other state along a path of length less or equal to the depth $\gamma.$ 
Therefore, for all $i=1, \dots, N,$  we have
\begin{equation} \label{final_above}
  \omega_{i}(t) \leq 
 M_0 \left [ 1-e^{-\kappa2\gamma\tau}\delta_{-}\xi_*^{\gamma-1} \left ( \frac{1-e^{-\frac {\kappa\tau} {N-1}}}{N-1} \right )^{\gamma-1}\right ], \quad t \in [t_*+(2\gamma-1)\tau, t_*+2\gamma\tau].
\end{equation}

Now, consider the state $R \in \{1,\dots,N\}$ such that $\omega_R(s)=M_0$ for some $s \in [t_*-\tau_R, t_*]$. Again, we can find a closed interval $[\alpha_R,\beta_R] \subset [t_*-\tau_R, t_*]$ such that
$$ \frac{m_0+M_0}{2}\leq\omega_R(t) \leq M_0, \ t \in [\alpha_R,\beta_R].$$
Using arguments analogous to the previous ones, we find a lower bound for all $\omega_{i}(t),$ $i=1, \dots, N$:
\begin{equation} \label{final_below}
\omega_i(t) \ge  m_0 \left [ 1+ e^{-\kappa2\gamma\tau} \delta_{+}\xi_*^{\gamma-1}\left ( \frac{1-e^{-\frac{\kappa\tau}{N-1}}}{N-1} \right )^{\gamma-1} \right ], \quad 
t \in [t_*+(2\gamma-1)\tau, t_*+2\gamma\tau],
\end{equation}
with
$$ \delta_{+}:=\frac{\xi_*}{2(N-1)}(1-e^{-\frac{\kappa}{N-1}\sigma})\Big(\frac{M_0}{m_0}-1\Big).$$
Finally, recalling definition \eqref{Gamma}, from \eqref{final_above} and \eqref{final_below}, 
we can find the final estimate 
$$ m_0+\frac{\Gamma}{2}(M_0-m_0)\leq  \omega_i(t) \leq M_0 - \frac{\Gamma}{2}(M_0-m_0), $$
for $i=1, \dots, N$ and $t \in [t_*+(2\gamma-1)\tau, t_*+2\gamma\tau]$.
\end{proof}

%
%
%
%
%
%
%
%
%

We are now able to prove the main result of this paper. 

\begin{theorem}\label{theorem_main2}
Let $\{\theta_i(t)\}_{i=1}^N$ be a solution to the Kuramoto model \eqref{main_eq2} on a strongly connected digraph $\mathcal{G}$, with initial data \eqref{IC2}. Assume that the hypotheses of Lemma \ref{Main_Strong} are satisfied.
Then, the Kuramoto oscillators with delayed coupling achieve the asymptotic complete frequency synchronization in the sense of Definition \ref{def_sync}.
\end{theorem}

\begin{proof}
We define the quantities $D_n:=M_n-m_n$ and
\begin{equation*}
 \Gamma_n:=\Big(\frac{\xi_*}{N-1}\Big)^{\gamma}e^{-\kappa2\gamma\tau}(1-e^{-\frac{\kappa \tau}{N-1}})^{\gamma-1}(1-e^{-\frac{\kappa \sigma_n}{N-1}}),
\end{equation*}
where $\sigma_n:= \min \{ \tau_0, \frac{M_n-m_n}{4\kappa R_\omega} \}.$ Then, $\Gamma_0$ is the constant $\Gamma$ defined in \eqref{Gamma} and $\Gamma_n \in (0,1)$  for all $n\ge 0.$ Now, we use Lemma \ref{lemma 3} on the intervals $J_n:=[t_*+(2\gamma n-1)\tau, t_*+ 2\gamma n\tau],$ $n\in  \mathbb{N}$.
For $t\in J_1$ we get
\begin{equation*}
\begin{split}
D_1&=M_1-m_1 \cr
&= \max_{i=1,\dots,N} \max_{t \in I_1^i} \  \omega_i(t) - \min_{i=1,\dots,N} \min_{t \in I_1^i} \ \omega_i(t) \\
& \leq M_0 - \frac{\Gamma_0}{2}(M_0-m_0)-m_0-\frac{\Gamma_0}{2}(M_0-m_0)\cr
&=(M_0-m_0)(1-\Gamma_0)\cr
&=D_0 (1-\Gamma_0).
\end{split}
\end{equation*}
Thus, we find that $D_1 \leq (1-\Gamma_0)D_0.$ Iterating the argument gives
$$ D_{n+1}\leq (1-\Gamma_n)D_n, \ \forall n \in \mathbb{N}.$$
Let us denote
$$ \tilde{\sigma}(D):=\min \Big\{ \tau_0, \frac{D}{4\kappa R_\omega} \Big\},$$
so that $\sigma_n = \tilde{\sigma}(D_n)$ for $n \in \mathbb{N}.$
Similarly, let us denote
$$ \tilde{\Gamma}(D):= \Big(\frac{\xi_*}{N-1}\Big)^{\gamma}e^{-\kappa2\gamma\tau}(1-e^{-\frac{\kappa \tau}{N-1}})^{\gamma-1}\left (1-e^{-\frac{\kappa \tilde{\sigma}(D)}{N-1}}\right )$$
so that $\Gamma_n = \tilde{\Gamma}(D_n)$ for $n \in \mathbb{N}.$ Then we find 
$$ D_{n+1}\leq (1-\tilde{\Gamma}(D_n))D_n,$$
and subsequently, $\{ D_n \}_{n \in \mathbb{N}}$ is a non-negative and decreasing sequence.   Passing to the limit $n\to\infty$ in the above estimate, and denoting $D$ the limit of $\{ D_n \}_{n \in \mathbb{N}},$ we have
$$ D \leq (1-\tilde{\Gamma}(D))D,$$
that is true only if $\tilde{\Gamma}(D) \leq 0.$ This gives $D=0$ and, noticing that, from Lemma \ref{5.1} and Remark \ref{remark5.1},  for further times to respect $J_n,$  we have
$\omega_i(t)-\omega_j(t) \leq M_n-m_n = D_n$ for all $i,j=1,\dots,N$, we conclude that
$$ |\omega_i(t)-\omega_j(t)| \rightarrow 0, $$
as $t \rightarrow + \infty$ and for all $i,j=1,\dots,N.$ 
Then, the system achieves asymptotic synchronization. \end{proof}

%
%
%
%
\section{Asymptotic synchronization: all-to-all connected case}\label{sec_syn_all2}
In this section, we consider the Kuramoto oscillators with delayed coupling in the case of all-to-all connection, i.e. $\chi_{ij}=1,$ for all $i,j=1, \dots, N$.  In this case, the system \eqref{main_eq2} reduces to 
\begin{equation}\label{main_eq}
\frac{d}{dt}\theta_i(t) = \Omega_i + \frac\kappa {N-1} \sum_{k\ne i} \sin(\theta_k(t - \tau_{ik}) - \theta_i(t)), \quad i = 1,\dots, N, \quad  t > 0,
\end{equation}
subject to the initial data \eqref{IC2}. 

Differently from the strongly connected case, in the all-to-all connected case, we prove the exponential frequency synchronization of system \eqref{main_eq}, i.e. the frequency-diameter decays to zero exponentially fast as time goes to infinity. 

To be more specific, our main result of this section is the following.
\begin{theorem}\label{theorem_main1} 
Let $\{\theta_i(t)\}_{i=1}^N$ be a global-in-time solution to the Kuramoto model \eqref{main_eq} with initial data \eqref{IC2}. Assume that the hypotheses of Lemma \ref{Main_Strong} are satisfied.
Then, the time-delayed Kuramoto oscillators achieve the asymptotic complete frequency synchronization in the sense of Definition \ref{def_sync} exponentially fast. More precisely, we have
\[
d_\omega(t) \leq {\mathcal C} e^{-\tilde \gamma t} \quad \mbox{for all} \quad t \geq 0,
\]
where $\gamma$ and ${\mathcal C}$ are positive constants depending on $R_\omega$, $\kappa$,  and $\tau$.
\end{theorem}

To prove Theorem \ref{theorem_main1}  we need some preliminary lemmas (cf. \cite{Rodriguez, CP23}).
In the whole section, we assume that the hypotheses of Lemma \ref{Main_Strong} are satisfied.

First, extending \eqref{Dzero}, we define, for $n\in\N,$ 
\begin{equation*}
\begin{array}{l}
\displaystyle{
D_\theta^*(n) :=\max_{\substack{1 \leq i,j \leq N, \\ s,t\in [t_*+(n-1)\tau, t_*+n\tau]}}|\theta_i(s) - \theta_j(t)|,}\\
\displaystyle{
D_\omega^*(n) :=\max_{\substack{1 \leq i,j \leq N, \\ s,t\in [t_*+(n-1)\tau, t_*+n\tau]}}|\omega_i(s) - \omega_j(t)|,}
\end{array}
\end{equation*}
where $t_*$ is the time in \eqref{diam_phase_small}.


%
%
%
%
%
%
%
%
%

The following lemma generalizes Lemma \ref{5.1}. We omit the proof since it is analogous to the previous one.
\begin{lemma}\label{3.1}
Let $\{\theta_i(t)\}_{i=1}^N$ be a solution to the Kuramoto model \eqref{main_eq} with initial conditions \eqref{IC2}. 
 Then,  for any $T \geq t_*$, we have
\begin{equation*}
\min_{j=1,...,N} \min_{s \in [T-\tau, T]}  \omega_j(s) \leq  \omega_i(t) \leq \max_{j=1,...,N} \max_{s \in [T-\tau, T]} \omega_j(s) 
\end{equation*}
for all $t \geq T - {\tau}$ and $i=1,...,N$.
\end{lemma}

From Lemma \ref{3.1}, we deduce the properties below.
\begin{lemma}\label{3.2}
For every $n \in \mathbb{N}$ and  $i,j=1,...,N,$ we get 
\begin{equation*}
\vert \omega_i(s)-\omega_j(t) \vert \leq D_{\omega}^*(n), \quad \forall s,t \geq t_*+(n-1)\tau.
\end{equation*}
\end{lemma}
\begin{proof}
Fix $n \in \mathbb{N}$ and $i,j=1,...,N.$ Without loss of generality, taken $s,t \geq t_*+(n-1)\tau$, we can suppose $\omega_i(s)>\omega_j(t).$
Then, using the Lemma \ref{3.1}   with $T=t_*+n\tau$ and Cauchy--Schwartz inequality, we have
\begin{align*}
\vert \omega_i(s)-\omega_j(t)\vert &=\omega_i(s)-\omega_j(t)\cr
&\leq \max_{l=1,...,N} \max_{r \in [t_*+(n-1)\tau, t_*+n\tau]} \omega_l(r)- \min_{l=1,...,N} \min_{r \in [t_*+(n-1)\tau, t_*+n\tau]}   \omega_l(r)\cr
&\leq \max_{l,k=1,...,N} \max_{r, \sigma \in [t_*+(n-1)\tau, t_*+n\tau]} \vert \omega_l(r)-\omega_k(\sigma) \vert= D_{\omega}^*(n).
\end{align*}
This completes the proof.
\end{proof}

\begin{remark}
Note that, from Lemma \ref{3.2}, for $n=0$, we find $\vert \omega_i(s)-\omega_j(t) \vert \leq D_{\omega}^*(0).$ Moreover, 
$$ D_{\omega}^*(n+1) \leq D_{\omega}^*(n), \ \forall\, n \in \mathbb{N}.$$
\end{remark}

\begin{lemma}\label{3.3}
For all $i,j=1,...,N $ and $n \in \mathbb{N},$ we have that 
\begin{equation} \label{l1}
\omega_i(t)- \omega_j(t)\leq e^{-\kappa(t-t_0)} ( \omega_i(t_0)-\omega_j(t_0)) + (1-e^{-\kappa(t-t_0)}) D_{\omega}^*(n),
\end{equation}
for all $t \geq t_0 \geq t_*+n \tau.$ Moreover, for all $s_0,t_0 \in [t_*+n\tau, t_*+(n+2)\tau]$, we obtain
\begin{equation} \label{l2}
\omega_i(t_*+(n+2)\tau)- \omega_j(t_*+(n+2)\tau) \leq e^{-2\kappa\tau}(\omega_i(t_0)-\omega_j(s_0))+ (1-e^{-2\kappa\tau}) D_{\omega}^*(n).
\end{equation}
\end{lemma}
\begin{proof}
Fix $n \in \mathbb{N}$ and let us denote 
\[
M_n^* := \max_{i=1,...,N} \max_{t \in [t_*+(n-1)\tau, t_*+n \tau]} \omega_i(t) \quad \mbox{and} \quad m_n^* := \min_{i=1,...,N} \min_{t \in [t_*+(n-1)\tau, t_*+n\tau]} \omega_i(t).
\]
Note that $ M_n^*-m_n^* \leq D_{\omega}^* (n)$. 
For all $i=1,...,N$ and $t \geq t_0 \geq t_*+n \tau $, we write
\begin{align*}
\begin{aligned}
\frac{d}{d t} \omega_i(t) & = \frac{\kappa}{N-1} \sum_{k \neq i} \cos (\theta_k(t-\tau_{ik})-\theta_i(t))( \omega_k(t-\tau_{ik}) -\omega_i(t)) \cr
& \leq \frac{\kappa}{N-1} \sum_{k \neq i} \cos (\theta_k(t-\tau_{ik})-\theta_i(t))(M_{n-1}^*-\omega_i(t)),
\end{aligned}
\end{align*}
where we used the fact that $t-\tau_{ik} \ge t_*+(n-1)\tau.$
By Lemma 3.1, since $t \geq t_*+n \tau$, we get $\omega_i(t) \leq M_{n-1}^*$, that implies $M_{n-1}^*- \omega_i(t)\geq 0$. Then we write 
$$ \frac{d}{d t} \omega_i(t) \leq \kappa (M_{n-1}^*- \omega_i(t)). $$
Applying Gr\"onwall's lemma, we find 
\begin{equation} \label{15}
\omega_i(t)\leq e^{-\kappa (t-t_0)} \omega_i(t_0) + (1-e^{-\kappa (t-t_0)})M_{n-1}^*.
\end{equation}
Similarly, for all $i=1,...,N$ and $t \geq t_0 \geq t_*+n \tau$, we deduce
\begin{equation} \label{16}
\omega_i(t) \geq e^{-\kappa (t-t_0)} \omega_i(t_0) + 
(1-e^{-\kappa(t-t_0)})m_{n-1}^*.
\end{equation}
Therefore, for $i,j=1,...,N$ and $t \geq t_0 \geq t_*+n \tau,$  we have 
\begin{align*}
\omega_i(t)-\omega_j(t) &\leq e^{-\kappa (t-t_0)} \omega_i(t_0) + (1-e^{-\kappa (t-t_0)}) M_{n-1}^* - e^{-\kappa (t-t_0)}  \omega_j(t_0) - (1-e^{-\kappa (t-t_0)})m_{n-1}^*\cr
&= e^{-\kappa (t-t_0)} ( \omega_i(t_0)-\omega_j(t_0)) + (1-e^{-\kappa (t-t_0)})(M_{n-1}^*-m_{n-1}^*)\cr
&\leq e^{-\kappa (t-t_0)}( \omega_i(t_0)-\omega_j(t_0)) + (1-e^{-\kappa (t-t_0)}) D_{\omega}^*(n).
\end{align*}
This gives the first assertion \eqref{l1}. 

For \eqref{l2}, from \eqref{15} with $t_0 \in [t_*+n \tau, t_*+(n+2)\tau]$ and $t:=t_*+(n+2)\tau$ we find
\begin{equation} \label{17}
\begin{split}
\omega_i(t_*+(n+2)\tau) & \leq e^{-\kappa (t_*+(n+2)\tau-t_0)} \omega_i(t_0) + (1-e^{-\kappa (t_*+(n+2)\tau-t_0)})M_{n-1}^* \\
& = e^{-\kappa (t_*+(n+2)\tau-t_0)}(\omega_i(t_0)-M_{n-1}^*)+M_{n-1}^* \\
& \leq e^{-2\kappa \tau}(\omega_i(t_0)-M_{n-1}^*)+M_{n-1}^* \\
& = e^{-2\kappa\tau}\omega_i(t_0) + (1-e^{-2\kappa \tau})M_{n-1}^*.
\end{split}
\end{equation}
Similarly, from \eqref{16}, for any $j=1,\dots,N$ and $s_0 \in [t_*+n\tau, t_*+(n+2)\tau]$ we have 
\begin{equation} \label{18}
\omega_i(t_*+(n+2)\tau)\geq e^{-2\kappa\tau}\omega_i(s_0)+ (1-e^{-2\kappa \tau})m_{n-1}^*.
\end{equation}
Combining \eqref{17} and \eqref{18} we arrive at \eqref{l2}, and this completes the proof.
\end{proof}

\begin{lemma} \label{4.4}
For all $n \in \mathbb{N},$ we have that 
\begin{equation*}
D_{\omega}^*(n+1) \leq e^{-\kappa \tau}d_\omega(t_*+n\tau) + (1-e^{-\kappa \tau})D_{\omega}^*(n).
\end{equation*}
\end{lemma}

\begin{proof}
Fix $ n \in \mathbb{N}$ and let $i,j=1,...,N,$ $s,t \in [t_*+n \tau,t_*+ (n+1)\tau],$ such that $D_{\omega}^*(n+1)= \vert \omega_i(s)-\omega_j(t) \vert$. Notice that if $\vert \omega_i(s)-\omega_j(t) \vert =0$ then it is obvious that  
\[
0 = D_{\omega}^*(n+1) \leq e^{-\kappa \tau}d_\omega(t_*+n \tau) + (1-e^{-\kappa \tau})D_{\omega}^*(n).
\]
Then, we now suppose $\vert \omega_i(s)-\omega_j(t) \vert >0$ and, without loss of generality, we may assume $\omega_i(s)>\omega_j(t).$
From \eqref{15} with $t_0 =t_*+ n \tau,$ we get
\begin{equation*}
\begin{aligned}
\omega_i(s) & \leq e^{-\kappa (s- t_*-n\tau)}\omega_i(t_*+n \tau) + (1-e^{-\kappa(s-t_*-n\tau)})M_{n}^* \cr
& = e^{-\kappa (s-t_*- n\tau)} (\omega_i(t_*+n \tau)-M_{n}^*) +M_{n}^*.
\end{aligned}
\end{equation*}
Thus, since $s \leq t_*+(n+1)\tau,$ we obtain 
$$\omega_i(s)\leq e^{-\kappa \tau} (\omega_i(t_*+n \tau) -M_{n}^*) +M_{n}^*= e^{-\kappa \tau} \omega_i(t_*+n \tau)+ (1-e^{-\kappa \tau})M_{n}^*. $$
Analogously, we have that 
$$ \omega_j(t)\geq e^{-\kappa \tau}\omega_j(t_*+n \tau) +(1-e^{-\kappa \tau})m_{n}^*. $$
Therefore, we conclude that 
$$
\begin{aligned}
D_{\omega}^*(n+1) & \leq e^{-\kappa \tau}(\omega_i(t_*+n \tau)- \omega_j(t_*+n \tau)) + (1-e^{-\kappa \tau}) D_{\omega}^*(n) \cr
& \leq e^{-\kappa \tau}d_\omega(t_*+n \tau) + (1-e^{-\kappa \tau}) D_{\omega}^*(n).
\end{aligned}
$$
\end{proof}

Before proving our main result we need a further lemma. 
\begin{lemma}\label{3.4}
There exists a constant $ C \in (0,1) $ such that 
$$ d_{\omega}(t_*+n \tau) \leq C D_{\omega}^*(n-2), $$
for all $n \geq 2$. 
\end{lemma}

\begin{proof}
If $d_{\omega}(t_*+n \tau)=0,$ then the assetion is obvious for all $C \in (0,1).$ Then, let us suppose $d_{\omega}(t_*+n \tau)>0$. Fix $i,j=1,...,N$ such that $d_{\omega}(t_*+n \tau)= |\omega_i(t_*+n \tau)-\omega_j(t_*+n \tau)|.$ Without lost of generality, we may assume $\omega_i(t_*+n \tau)>\omega_j(t_*+n \tau).$ 
As before, we consider the following quantities:
$$ M_{n-1}^*= \max_{l=1,...,N} \max_{s \in [t_*+(n-2)\tau, t_*+(n-1)\tau]}  \omega_l(s) $$
and 
$$ m_{n-1}^*= \min_{l=1,...,N} \min_{s \in [t_*+(n-2)\tau,t_*+ (n-1)\tau]}  \omega_l(s)\,. $$
Then we get $M_{n-1}^*-m_{n-1}^*\leq D_{\omega}^*(n-1)$.

Let us distinguish two cases. 

\textbf{Case I}: Suppose that there exist $t_0,s_0 \in [t_*+(n-2)\tau, t_*+ n\tau]$ such that $\omega_i(t_0)-\omega_j(s_0)< 0$. Then, due to \eqref{l2} we have
\begin{align*}
\begin{aligned}
d_{\omega}(t_*+n \tau) & \leq e^{-2\kappa\tau}(\omega_i(t_0)-\omega_j(s_0))+ (1-e^{-2\kappa\tau})D_{\omega}^*(n-2) \cr
& \leq (1-e^{-2\kappa\tau})D_{\omega}^*(n-2).
\end{aligned}
\end{align*}
Then, the assertion follows in this case. 

\textbf{Case II}: Suppose that $ \omega_i(t)-\omega_j(s)\geq 0$ for all $t,s \in  [t_*+(n-2)\tau,t_*+ n\tau] $. Then, for $t \in [t_*+(n-1)\tau, t_*+n \tau]$ we obtain
\begin{align}
\begin{aligned}
\frac{d}{d t}( \omega_i(t)- &\omega_j(t)) = \frac{\kappa}{N-1} \sum_{k \neq i } \cos (\theta_k(t-\tau_{ik})-\theta_i(t))(\omega_k(t-\tau_{ik})-\omega_i(t))\cr
& - \frac{\kappa}{N-1} \sum_{k \neq j } \cos (\theta_k(t-\tau_{jk})-\theta_j(t)) (\omega_k(t-\tau_{jk})-\omega_j(t))\cr
& = \frac{\kappa}{N-1} \sum_{k \neq i } \cos (\theta_k(t-\tau_{ik})-\theta_i(t)) (\omega_k(t-\tau_{ik})-M_{n-1}^*+M_{n-1}^*- \omega_i(t)) \cr
& + \frac{\kappa}{N-1} \sum_{k \neq j } \cos (\theta_k(t-\tau_{jk})-\theta_j(t)) (\omega_j(t)-m_{n-1}^*+m_{n-1}^*-\omega_k(t-\tau_{jk})) \cr
& := S_1+S_2.
\end{aligned} \label{eq}
\end{align}
Since $t \in [t_*+(n-1)\tau,t_*+ n \tau]$, we get $ t-\tau_{ij} \in [t_*+(n-2)\tau,t_*+ (n-1) \tau]$ for all $i,j=1,\dots,N$. Thus, it follows from Lemma \ref{3.1} with $T=t_*+(n-2)\tau$ that 
\[
m_{n-1}^* \leq \omega_k(t)\leq M_{n-1}^* \quad \mbox{and} \quad m_{n-1}^* \leq \omega_k(t-\tau_{ik})\leq M_{n-1}^*
\] 
for all $i,k=1,...,N.$ Then, we estimate
\begin{align}
\begin{aligned}
S_1 & = \frac{\kappa}{N-1} \sum_{k \neq i } \cos (\theta_k(t-\tau_{ik})-\theta_i(t)) (\omega_k(t-\tau_{ik}) -M_{n-1}^*) \cr
& \quad + \frac{\kappa}{N-1} \sum_{k \neq i } \cos (\theta_k(t-\tau_{ik})-\theta_i(t)) (M_{n-1}^*-\omega_i(t)) \cr
& \leq \frac{\kappa}{N-1} \xi_* \sum_{k \neq i}  (\omega_k(t-\tau_{ik})-M_{n-1}^*) + \kappa  (M_{n-1}^*-\omega_i(t)),
\end{aligned} \label{s1}
\end{align}
and
\begin{align}
\begin{aligned}
S_2 & := \frac{\kappa}{N-1} \sum_{k \neq j } \cos (\theta_k(t-\tau_{jk})-\theta_j(t)) (\omega_j(t)-m_{n-1}^*) \cr
& \quad + \frac{\kappa}{N-1} \sum_{k \neq j } \cos (\theta_k(t-\tau_{jk})-\theta_j(t)) (m_{n-1}^*- \omega_k(t-\tau_{jk})) \cr
& \leq \kappa (\omega_j(t) -m_{n-1}^*)  + \frac{\kappa}{N-1} \xi_* \sum_{k \neq j} (m_{n-1}^*- \omega_k(t-\tau_{jk})),
\end{aligned} \label{s2}
\end{align}
where $\xi_*$ is the positive constant given in \eqref{lower bound}. 
Putting \eqref{s1} and \eqref{s2} in \eqref{eq}, we have
\begin{align*}
\begin{aligned}
\frac{d}{d t} (\omega_i(t)- &\omega_j(t))\leq \kappa (M_{n-1}^*-m_{n-1}^*) - \kappa (\omega_i(t)-\omega_j(t))\cr
& + \frac{\kappa}{N-1} \xi_* \sum_{k \neq i}  (\omega_k(t-\tau_{ik})-M_{n-1}^*) \cr
& + \frac{\kappa}{N-1} \xi_* \sum_{k \neq j} (m_{n-1}^*- \omega_k(t-\tau_{jk})).
\end{aligned}
\end{align*}

Notice that, since $\omega_k(t-\tau_{ik})-M_{n-1}^* \leq 0$ for all $i,k=1,\dots,N$ and $t \in [t_*+(n-1)\tau, t_*+n\tau]$ and $j \neq i,$ we can write 
$$ \sum_{k \neq i} (\omega_k(t-\tau_{ik}) -M_{n-1}^*) \leq \omega_j(t-\tau_{ij}) -M_{n-1}^*. $$
Analogously, 
$$ \sum_{k \neq j} (m_{n-1}^*-\omega_k(t-\tau_{jk})) \leq m_{n-1}^*-\omega_i(t-\tau_{ji}). $$
Therefore, we obtain
\begin{equation*}
\begin{split}
\frac{d}{d t}(\omega_i(t)- \omega_j(t))&  \leq \kappa \Big(1-\frac{\xi_*}{N-1}\Big)(M_{n-1}^*-m_{n-1}^*) - \kappa(\omega_i(t)-\omega_j(t)) \\
&  + \frac{\kappa}{N-1} \xi_* (\omega_j(t-\tau_{ij})- \omega_i(t-\tau_{ji})) \\
& \leq \kappa \Big(1-\frac{\xi_*}{N-1}\Big)(M_{n-1}^*-m_{n-1}^*) - \kappa ( \omega_i(t)-\omega_j(t)),
\end{split}
\end{equation*}
where we used the assumption $ \omega_i(t)-\omega_j(s)\geq 0$ for all $t,s \in  [t_*+(n-2)\tau,t_*+ n\tau] $. 

Applying now the Gr\"onwall's lemma over $[t, t_*+n\tau],$ with $t \in [t_*+(n-1)\tau, t_*+n\tau],$ we deduce
\begin{equation*}
\begin{split}
\omega_i(t)-\omega_j(t)& \leq e^{-\kappa(t-t_*-(n-1)\tau)}( \omega_i(t_*+(n-1)\tau)-\omega_j(t_*+(n-1)\tau)) \\
& + \Big( 1- \frac{\xi_*}{N-1} \Big) (M_{n-1}^*-m_{n-1}^*)(1-e^{-\kappa(t-t_*-(n-1)\tau)}).
\end{split}
\end{equation*}
For $t=t_*+n\tau,$ we get
\begin{equation*}
\begin{split}
d_{\omega}(t_*+n\tau)& \leq e^{-\kappa\tau}(\omega_i(t_*+(n-1)\tau)-\omega_j(t_*+(n-1)\tau))\\
&\quad
+ \Big( 1- \frac{\xi_*}{N-1} \Big) (M_{n-1}^*-m_{n-1}^*)(1-e^{-\kappa \tau}) \\
& \leq \Big( 1- \frac{\xi_*}{N-1}(1-e^{-\kappa \tau})\Big)D_{\omega}^*(n-2),
\end{split}
\end{equation*}
where we used that $M_{n-1}^*-m_{n-1}^* \leq D_{\omega}^*(n-1)$ and the monotonicity property of $D_{\omega}^*(n).$ 

Finally, we set
\[
C:= \max \Big \{ 1-e^{-2\kappa\tau}, 1-\frac{\xi_*}{N-1}(1-e^{-\kappa \tau}) \Big \} > 0
\]
to conclude the desired result.
\end{proof}

Now, we are ready to prove the exponential synchronization estimate for the system  \eqref{main_eq} with \eqref{IC2}. \\
\begin{proof}[Proof of Theorem \ref{theorem_main1}]
Let $\{\theta_i(t)\}_{i=1}^N$ a global classical solution of  \eqref{main_eq}. We claim that
$$ D_{\omega}^*(n+1) \leq \tilde{C} D_{\omega}^*(n-2), \quad \forall \,n \geq 2, $$
for a suitable constant $\tilde{C} \in (0,1).$
Indeed, for $n\geq 2$, applying Lemmas \ref{4.4} and \ref{3.4},  we obtain
\begin{align*}
\begin{aligned}
D_{\omega}^*(n+1) & \leq e^{-\kappa \tau} d_\omega(t_*+n \tau)+(1-e^{-\kappa \tau})D_{\omega}^*(n) \cr
& \leq e^{-\kappa \tau} C D_{\omega}^*(n-2) +  (1-e^{-\kappa \tau})D_{\omega}^*(n) \cr
&  \leq e^{-\kappa \tau} C D_{\omega}^*(n-2) + (1-e^{-\kappa \tau})D_{\omega}^*(n-2) \cr
& = D_{\omega}^*(n-2) [1-e^{-\kappa \tau}(1-C) ].
\end{aligned}
\end{align*}
Denoting $$ \tilde{C}:=  [1-e^{-\kappa \tau}(1-C) ], $$ 
we have the claim. This implies that 
$$ D_{\omega}^*(3n) \leq \tilde{C}^n D_{\omega}^*(0), \quad \forall \, n \geq 1. $$
Then, we have that 
$$ D_{\omega}^*(3n) \leq e^{-3n \tau\gamma} D_{\omega}^*(0), \quad \forall \, n \in \mathbb{N}, $$
with  
$$ \gamma := \frac{1}{3\tau} \ln\left (\frac{1}{\tilde{C}}\right ). $$
Now, fix 
$i,j=1,...,N.$ For all $t\geq t_*-\tau,$ we get $t \in [t_*+3n\tau-\tau, t_*+3n\tau+2\tau]$ for some $n \in \mathbb{N}$. Then, by applying Lemma \ref{3.2}, we can write 
$$ \vert \omega_i(t) - \omega_j(t) \vert \leq D_{\omega}^*(3n) \leq  e^{-3n \tau\gamma} D_{\omega}^*(0). $$
Since $t \leq t_*+3n \tau+2\tau,$ we have that 
$$ \vert \omega_i(t) - \omega_j(t) \vert \leq e^{-\gamma t} e^{\gamma (t_*+2\tau)} D_{\omega}^*(0).  $$
Finally, passing to the maximum on $i,j=1,...,N,$ we find that
$$ d_\omega(t) \leq e^{-\gamma(t-t_*-2\tau)}D_{\omega}^*(0), $$
and this completes the proof.
\end{proof}

%
%
%
%
%
%
%
%
\section{Numerical simulations}\label{sec_numer}
In this section, we present numerical experiments concerning the dynamics of solutions to the system \eqref{main_eq2} focusing on how these dynamics depend on the network structure and the magnitude of time delays. Specifically, we perform simulations for the all-to-all and strongly coupled system of 10 oscillators. We note that we do not seek to satisfy every condition of our main theorems due to their restrictive nature; instead, we perform simulations with more natural choices of parameters to illustrate the behavior of the model under varying couplings and time delays. 

For all of the numerical simulations, the natural frequencies were generated according to a zero-centered uniform distribution, rounded to three decimal places, yielding the following values:
\[
\Omega \approx (-0.563,  0.839, -0.119 ,  0.904 , -0.49349812,-0.063,  0.603,  0.979, -0.101, -0.060).
\]
The initial phases were set to be constant and drawn from a random uniform distribution in the half circle:
\[
\theta_0(s) \approx (1.714, 2.892, 2.684, 1.543, 1.081, 0.007, 2.025, 1.012, 1.228, 1.955),\quad s \in -[\tau,0].
\]

We utilized the JiTCDDE package from the JiTC$^*$DE toolbox \cite{Ans18} for Python, which provides an extension to the commonly used SciPy ODE, allowing for the simulation of delay differential equations.

We consider a standardized set of time delays $\tau_{ij} \in [0,1]$, which may then be scaled as we desired:
\[
\left(\tau_{ij}\right) = \begin{pmatrix}
0.941 & 0.432 & 0.440 & 0.953 & 0.497 & 0.126 & 0.941 & 0.501 & 0.361 & 0.901 \\
0.017 & 0.948 & 0.088 & 0.075 & 0.942 & 0.478 & 0.180 & 0.982 & 0.335 & 0.941 \\
0.459 & 0.368 & 0.573 & 0.911 & 0.980 & 0.696 & 0.368 & 0.864 & 0.924 & 0.543 \\
0.712 & 0.871 & 0.684 & 0.650 & 0.381 & 0.600 & 0.243 & 0.120 & 0.259 & 0.446 \\
0.345 & 0.702 & 0.383 & 0.505 & 0.291 & 0.089 & 0.717 & 0.224 & 0.465 & 0.594 \\
0.673 & 0.415 & 0.279 & 0.893 & 0.644 & 0.248 & 0.768 & 0.537 & 0.420 & 0.097 \\
0.222 & 0.105 & 0.697 & 0.748 & 0.436 & 0.054 & 0.170 & 0.945 & 0.892 & 0.398 \\
0.141 & 0.443 & 0.434 & 0.020 & 0.719 & 0.657 & 0.587 & 0.807 & 0.821 & 0.214 \\
0.070 & 0.369 & 0.881 & 0.776 & 0.255 & 0.733 & 0.567 & 0.276 & 0.721 & 0.284 \\
0.342 & 0.202 & 0.558 & 0.448 & 0.632 & 0.011 & 0.406 & 0.038 & 0.305 & 0.086
\end{pmatrix}.
\]
%
%
%
%
%
%
%
%
\subsection{All-to-all connected case}

 We first consider the all-to-all connected case with a coupling strength of $\kappa = 2$. To observe the effect of time delay, we use different values of time delays $\tau = 0$ and $\tau = 4.91$, obtained by scaling the standardized delays $(\tau_{ij})$ by a factor of 5. In Figure \ref{fig:a2a_t0}, we show the time evolution of phases and frequencies, and the diameters of phases and frequencies are depicted in \ref{fig:a2a_t02}. In particular, the logarithmic scale plot displays the exponential decay rate of convergence of the frequency diameter as expected by Theorem \ref{theorem_main1}.

\begin{figure}[H]
    \centering
    \includegraphics[width = \linewidth]{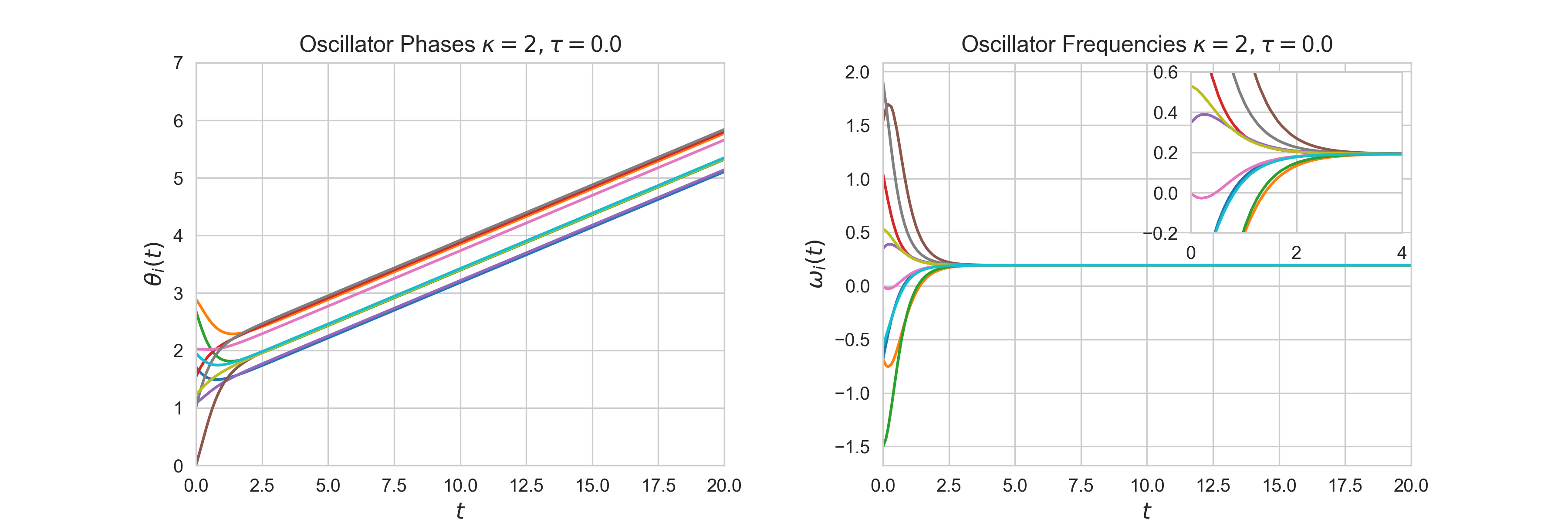}
    \caption{Time evolution of phases and frequencies in the case of all-to-all coupling and no time delay}
    \label{fig:a2a_t0}
\end{figure}
\begin{figure}[H]
    \centering
    \includegraphics[width = \linewidth]{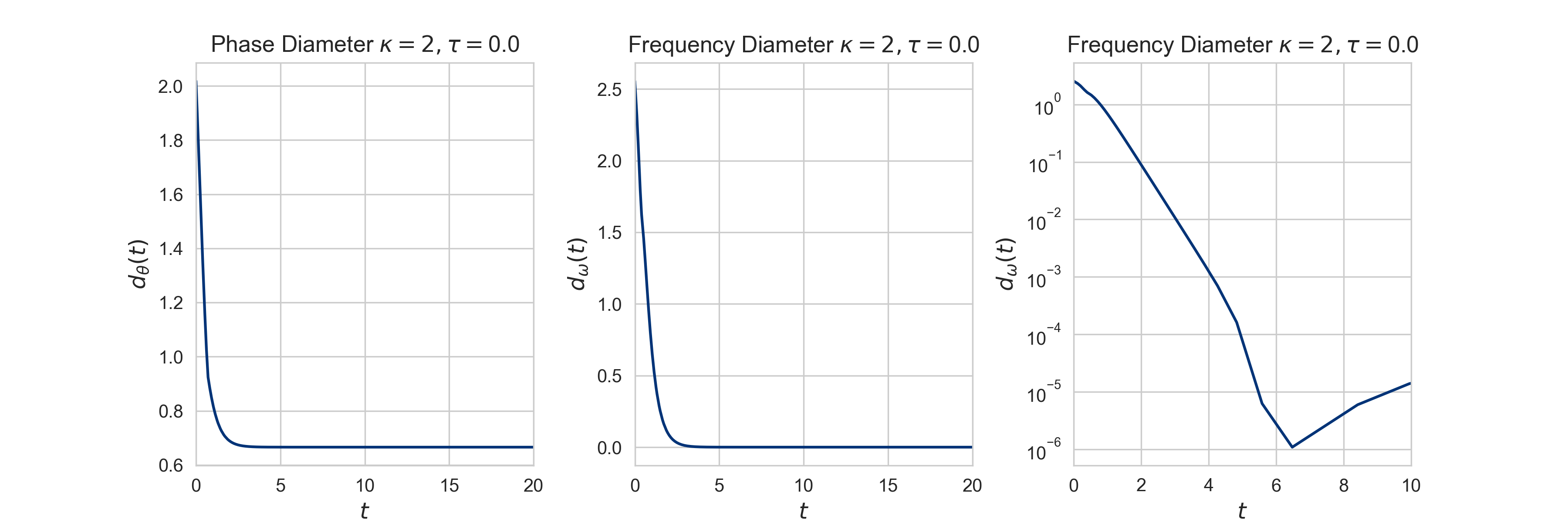}
    \caption{Time evolution of diameters of phases and frequencies in the case of all-to-all coupling and no time delay}   
    \label{fig:a2a_t02}
\end{figure}
 
On the other hand, for the large time delay, $\tau = 4.91$, we observe the oscillatory behavior of solutions and the frequency synchronization occurs at later times compared to the no delay case, see Figures \ref{fig:a2a_t4} and \ref{fig:a2a_t42}. 
\begin{figure}[H]
    \centering
    \includegraphics[width = \linewidth]{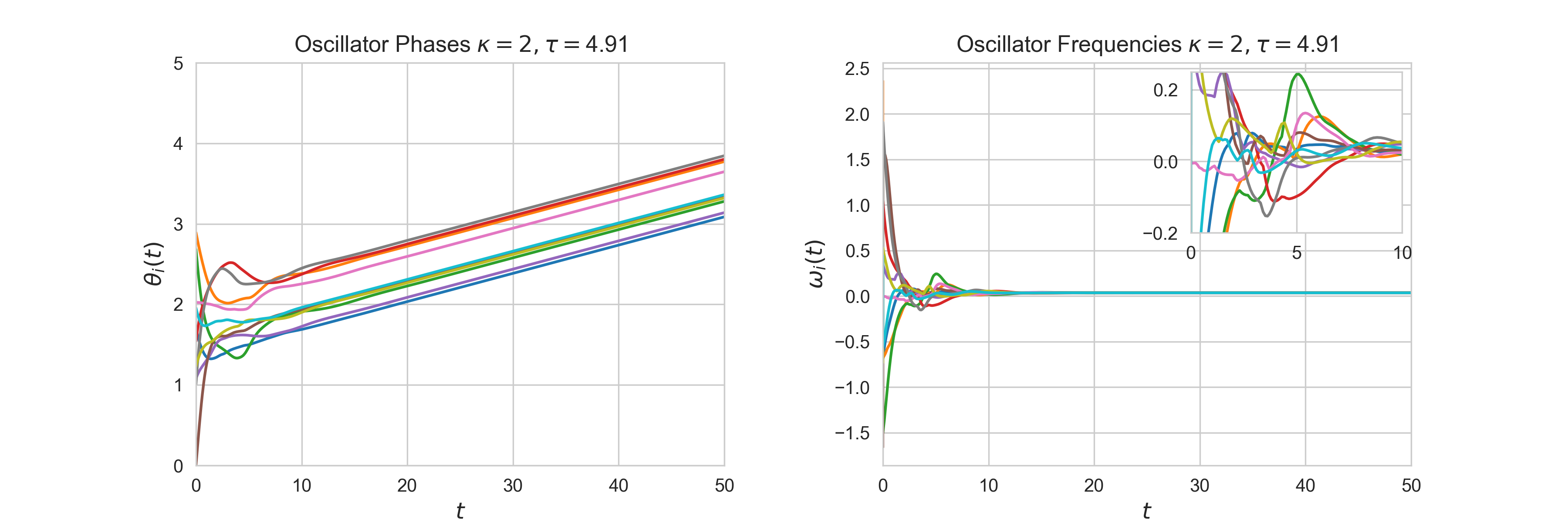}
    \caption{Time evolution of phases and frequencies in the case of all-to-all coupling and time delay}
    \label{fig:a2a_t4}
\end{figure}
\begin{figure}[H]
    \centering
    \includegraphics[width = \linewidth]{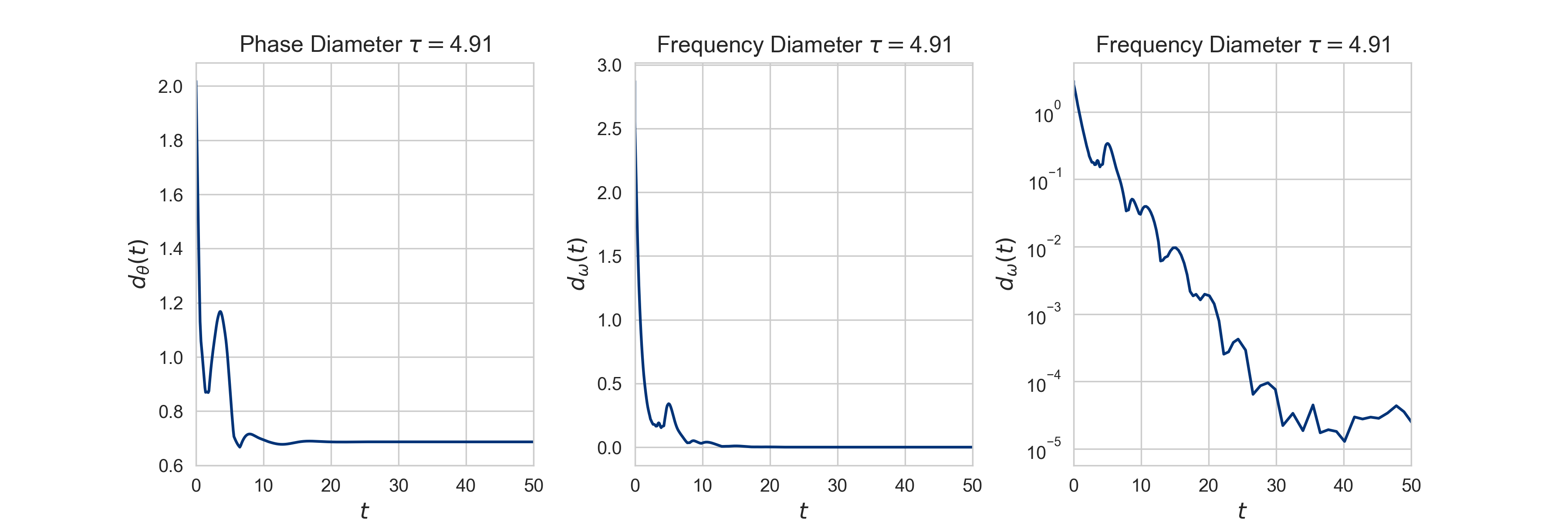}
     \caption{Time evolution of diameters of phases and frequencies in the case of all-to-all coupling and time delay}   
    \label{fig:a2a_t42}
\end{figure}

In both cases, $\tau  = 0$ and $\tau = 4.91$, we achieve synchronization due to the strong connectivity provided by all-to-all coupling with strong coupling strength. The simulations demonstrate how the introduction of time delay affects the synchronization process, leading to delayed but eventual synchronization in the presence of larger delays. 

%
%
%
%
%
%
%
%
\subsection{Strongly connected case}
Next, we investigate the dynamics of solutions in the case of a strongly connected digraph. Precisely, we choose a simple cyclic graph, as depicted in Figure \ref{fig:digraph}. This setup provides a different network structure compared to the all-to-all coupling and allows us to examine how weaker connectivity influences synchronization behavior.

\begin{figure}[H]
  \centering
  \begin{tikzpicture}[->, >=stealth, auto, node distance=2cm, thick, main node/.style={circle, draw, font=\sffamily\Large\bfseries}]

    \begin{scope}[local bounding box=graph]
      \def\numNodes{10}
      \def\circleRadius{3cm}
      \pgfmathtruncatemacro{\angleStep}{360/\numNodes}

      \foreach \x in {1,...,\numNodes}{
        \pgfmathtruncatemacro{\angle}{\angleStep*\x}
        \node[main node] (\x) at (\angle:\circleRadius) {\x};
      }

      \draw[->] (1) to (2);
      \draw[->] (2) to (3);
      \draw[->] (3) to (4);
      \draw[->] (4) to (5);
      \draw[->] (5) to (6);
      \draw[->] (6) to (7);
      \draw[->] (7) to (8);
      \draw[->] (8) to (9);
      \draw[->] (9) to (10);
      \draw[->] (10) to  (1);
    \end{scope}

    \node[right=0.75cm of graph] (matrixlabel) {$\chi = $};
   \matrix [matrix of math nodes,left delimiter=(,right delimiter=),right=of graph] (m)
    {
      0 & 1 & 0 & 0 & 0 & 0 & 0 & 0 & 0 & 0 \\
      0 & 0 & 1 & 0 & 0 & 0 & 0 & 0 & 0 & 0 \\
      0 & 0 & 0 & 1 & 0 & 0 & 0 & 0 & 0 & 0 \\
      0 & 0 & 0 & 0 & 1 & 0 & 0 & 0 & 0 & 0 \\
      0 & 0 & 0 & 0 & 0 & 1 & 0 & 0 & 0 & 0 \\
      0 & 0 & 0 & 0 & 0 & 0 & 1 & 0 & 0 & 0 \\
      0 & 0 & 0 & 0 & 0 & 0 & 0 & 1 & 0 & 0 \\
      0 & 0 & 0 & 0 & 0 & 0 & 0 & 0 & 1 & 0 \\
      0 & 0 & 0 & 0 & 0 & 0 & 0 & 0 & 0 & 1 \\
      1 & 0 & 0 & 0 & 0 & 0 & 0 & 0 & 0 & 0 \\
    };
  \end{tikzpicture}
  \caption{Strongly connected digraph and its adjacency matrix}
  \label{fig:digraph}
\end{figure}
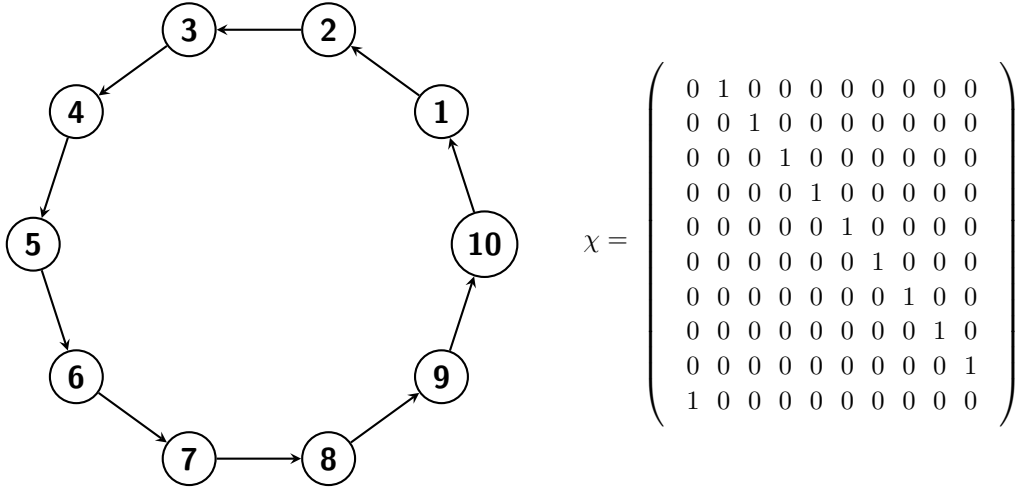

Note that in this case, the system \eqref{main_eq2} reduces to the following time-delayed Kuramoto oscillators unidirectionally coupled in a ring:
\begin{align*}
\dot \theta_i(t) &= \Omega_i + \frac\kappa {N-1}   \sin(\theta_{i+1}(t - \tau_{i (i+1)}) - \theta_i(t)), \quad i = 1,\dots, N-1,\cr
\dot \theta_N(t) &= \Omega_i + \frac\kappa {N-1}   \sin(\theta_1(t - \tau_{N 1}) - \theta_N(t)).
\end{align*}

As illustrated in Figures \ref{fig:sc_t0} and \ref{fig:sc_t02}, when the coupling strength is set to $\kappa = 2$, the system does not achieve frequency synchronization. This lack of synchronization is due to the weaker connectivity among the oscillators compared to the all-to-all coupling case, which indicates that the network structure significantly impacts the synchronization process.

\begin{figure}[H]
    \centering
    \includegraphics[width = \linewidth]{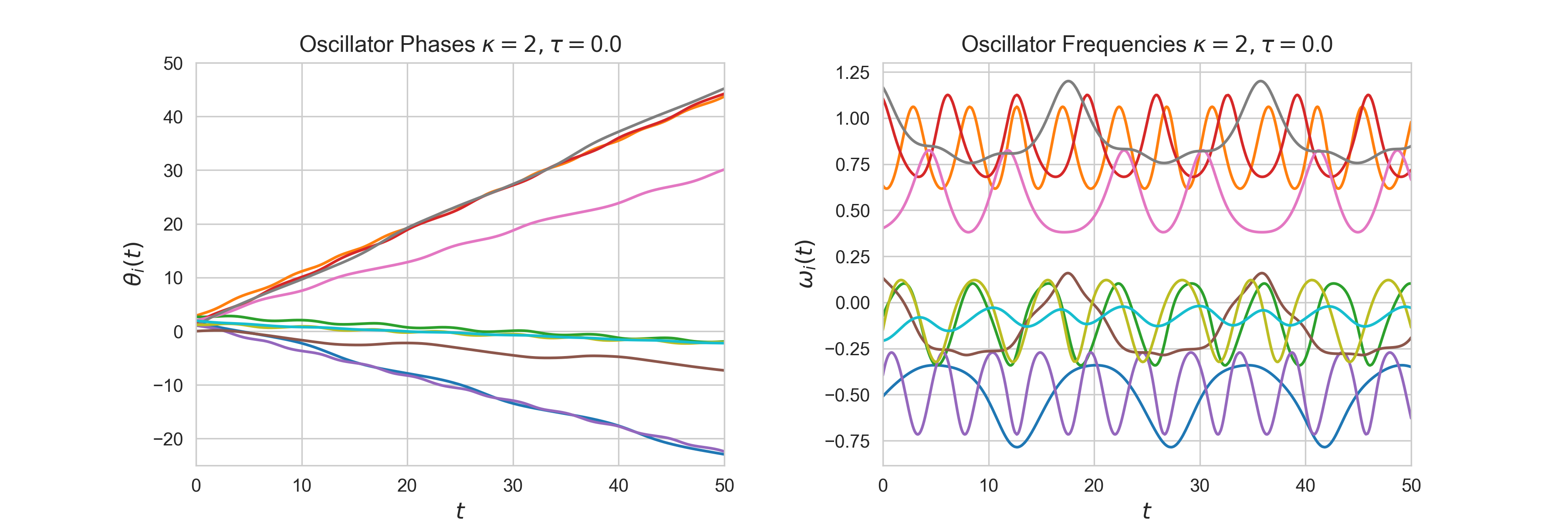}
    \caption{Time evolution of phases and frequencies in the case of strongly connected and no time delay $(\kappa = 2)$}
    \label{fig:sc_t0}
\end{figure}
\begin{figure}[H]
    \centering
    \includegraphics[width = \linewidth]{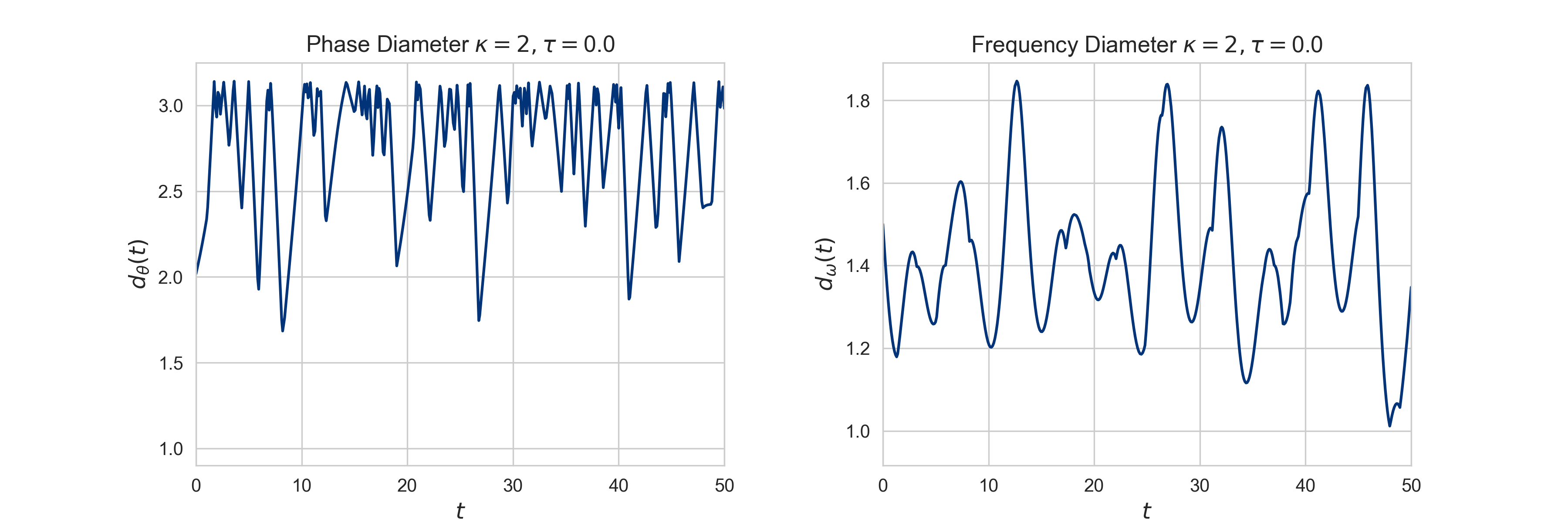}
    \caption{Time evolution of diameters of phases and frequencies in the case of strongly connected and no time delay $(\kappa = 2)$}
    \label{fig:sc_t02}
\end{figure}

On the other hand, increasing the coupling strength to $\kappa = 8$ results in frequency synchronization, as shown in Figures \ref{fig:k8t0} and \ref{fig:k8t02}. Despite achieving synchronization, we observe that the solutions exhibit oscillatory behavior before the synchronization. The synchronization occurs at much later times compared to the all-to-all coupling case, even with the stronger coupling strength. This indicates that stronger coupling is required in a less connected network to achieve similar synchronization results. We also note that in the case of synchronization, the decay rate of the frequency diameter is exponential, with some oscillatory perturbations. 

\begin{figure}[H]
    \centering
    \includegraphics[width = \linewidth]{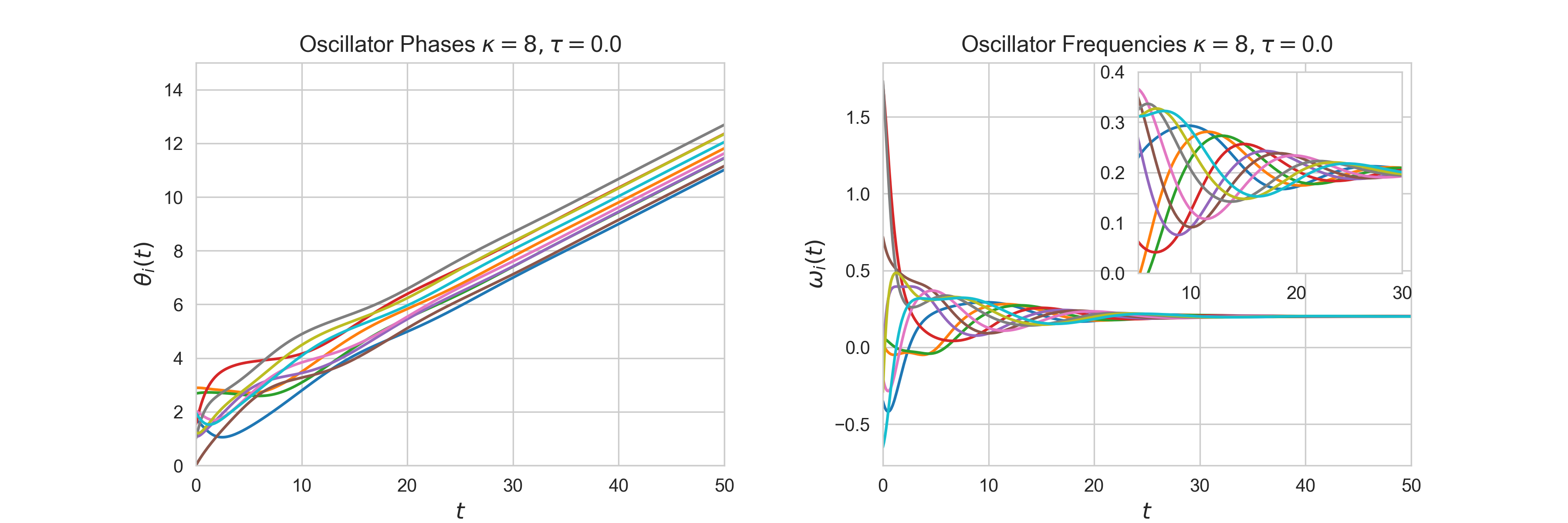}
    \caption{Time evolution of phases and frequencies in the case of strongly connected and no time delay $(\kappa = 8)$}
    \label{fig:k8t0}
\end{figure}
\begin{figure}[H]
    \centering
    \includegraphics[width = \linewidth]{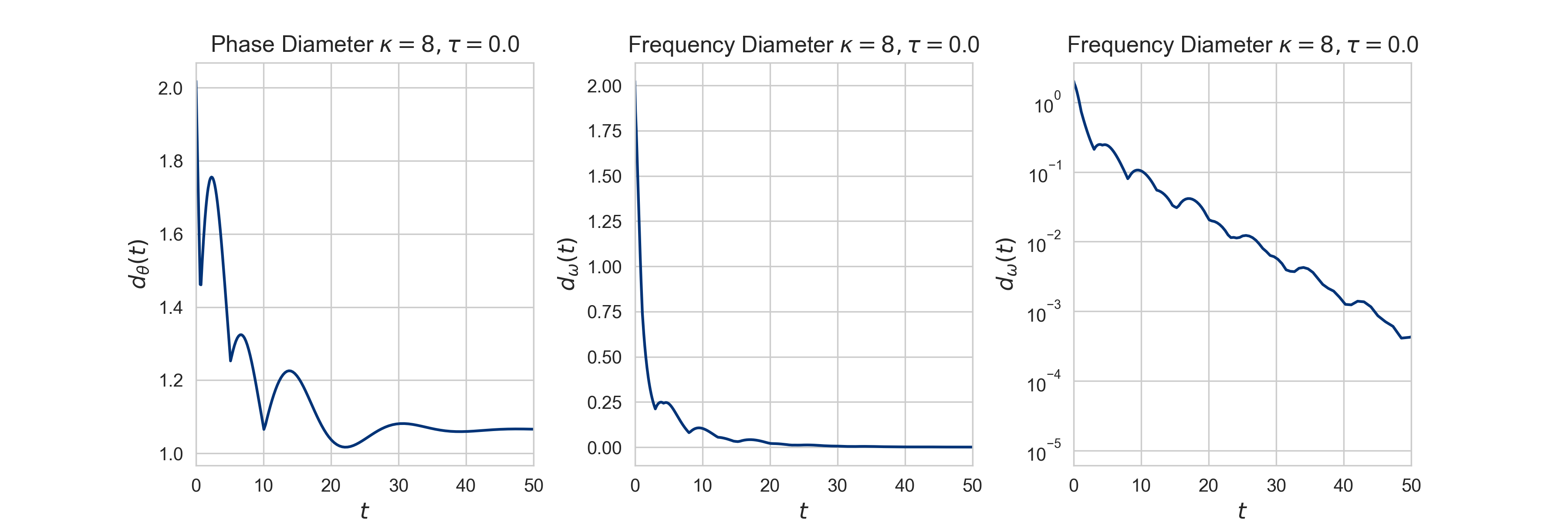}
        \caption{Time evolution of diameters of phases and frequencies in the case of strongly connected and no time delay $(\kappa = 8)$}
    \label{fig:k8t02}
\end{figure}
 
Finally, we examine the effect of large time delays by scaling the standardized delays $(\tau_{ij})$ by a factor of 30, resulting in $\tau = 29.47$. Figures \ref{fig:k8t29} and \ref{fig:k8t292} show the time evolution of phases and frequencies, along with their diameters. In this setting, we observe chaotic-type dynamics due to the combined effect of weak connectivity and strong time delays. The synchronization behavior cannot be achieved under these conditions, highlighting the critical influence of time delays on the dynamics of the system.

\begin{figure}[H]
    \centering
    \includegraphics[width = \linewidth]{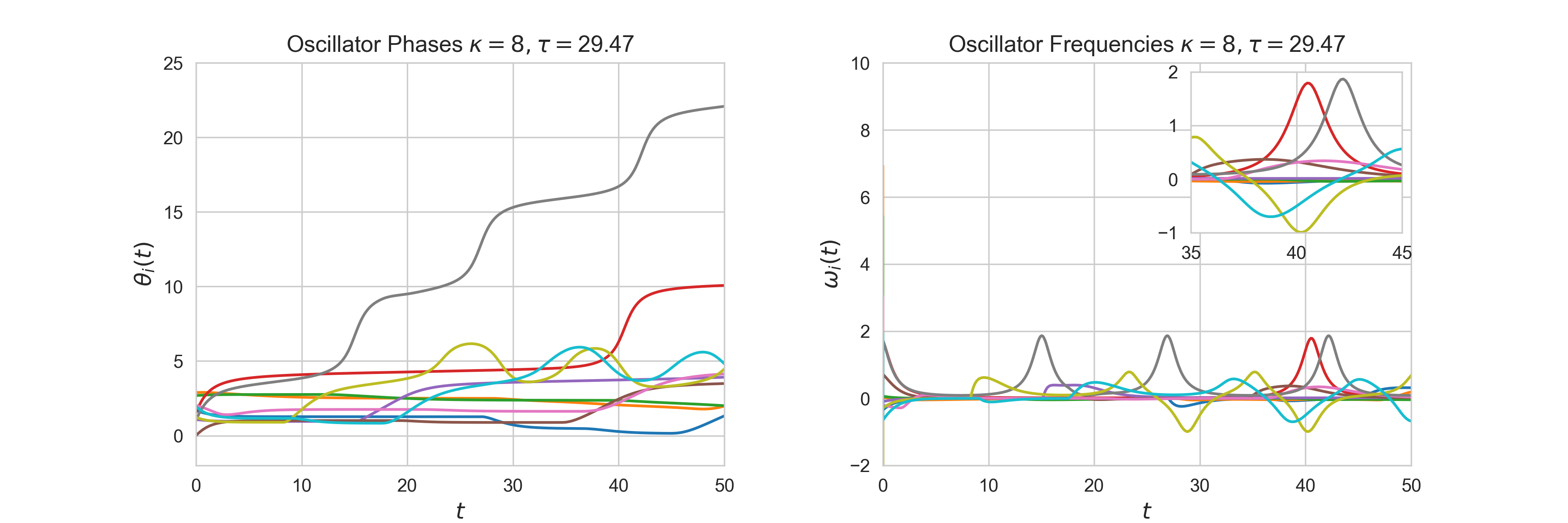}
    \caption{Time evolution of phases and frequencies in the case of strongly connected and time delay $(\kappa = 8)$}
    \label{fig:k8t29}
\end{figure}
\begin{figure}[H]
    \centering
    \includegraphics[width = \linewidth]{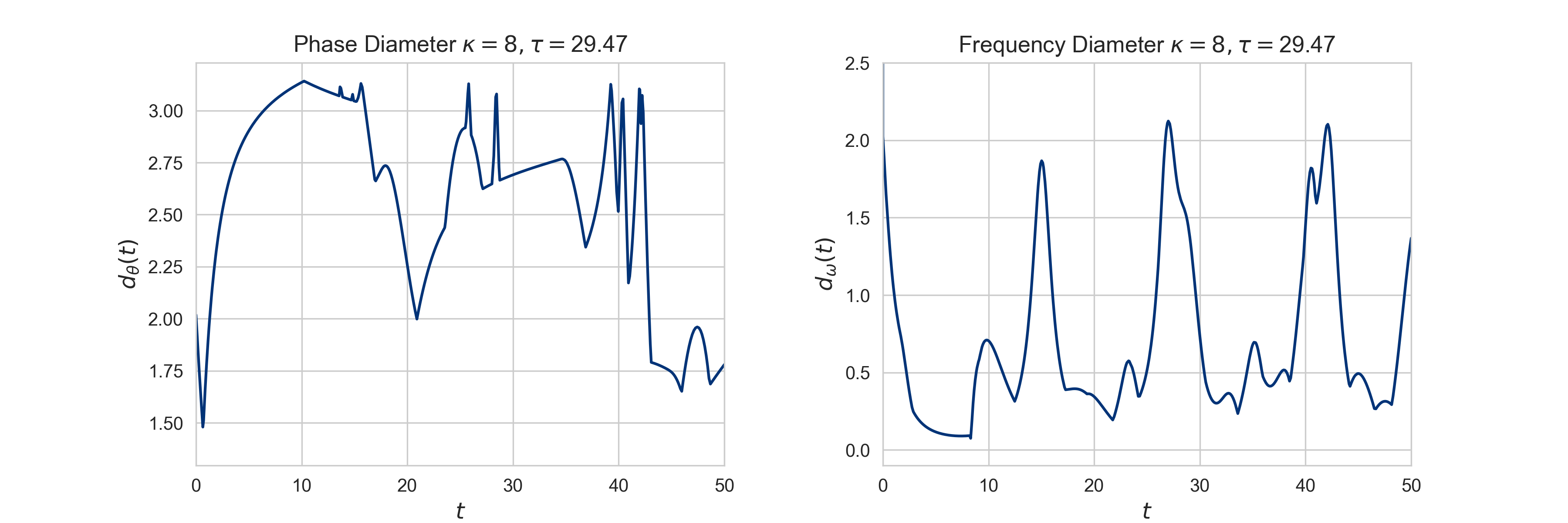}
            \caption{Time evolution of diameters of phases and frequencies in the case of strongly connected and time delay $(\kappa = 8)$}
    \label{fig:k8t292}
\end{figure}

In summary, our numerical experiments demonstrate that network structure and time delays significantly affect the dynamics of coupled oscillators. In strongly connected digraphs, stronger coupling strengths are required to achieve synchronization compared to all-to-all connected case. Moreover, large time delays introduce complex dynamics and hinder synchronization, emphasizing the delicate balance between coupling strength, network topology, and time delays in such systems.
%
%
%
%
%
%
%
%
\section*{Acknowledgments}
The work of Y.-P. Choi is supported by NRF grant no. 2022R1A2C1002820, and
the work of C. Pi\-gnot\-ti is partially
supported by PRIN 2022  (2022238YY5) {\it Optimal control problems: analysis,
approximation and applications}, by
PRIN-PNRR 2022 (P20225SP98) {\it Some mathematical approaches to climate change and its impacts}, and by INdAM GNAMPA Project {\it ``Modelli alle derivate parziali per interazioni multiagente non 
simmetriche"}(CUP E53C23001670001).

%
%
%
%

\end{document}